\documentclass[12pt]{amsart}
\usepackage{url,amssymb}
\usepackage[all]{xy}
\usepackage{fullpage}
\usepackage{color}
\usepackage[pdftex, hidelinks=true]{hyperref}
\usepackage{enumitem}
\usepackage{lscape}
\usepackage{graphicx}
\usepackage{caption}

\newif\ifdebug                                                      %
\debugfalse
\newcommand{\printname}[1]
   {\smash{\makebox[0pt]{\hspace{-.0in}\raisebox{8pt}{\tiny #1}}}}
\newcommand{\labell}[1] {\ifdebug {\label{#1}\printname{}\printname{#1}}
                        \else    {\label{#1}} \fi}

\def \ol {\overline}
\def \bm {\mathbf}

\swapnumbers
\numberwithin{equation}{section}

\newtheorem{Theorem}[equation]{Theorem}
\newtheorem{prop}[equation]{Proposition}
\newtheorem{Proposition}[equation]{Proposition}

\newtheorem{lemma}[equation]{Lemma}
\newtheorem*{lemma*}{Lemma}
\newtheorem{Lemma}[equation]{Lemma}

\newtheorem{Corollary}[equation]{Corollary}

\theoremstyle{definition}

\newtheorem{Definition}[equation]{Definition}
\newtheorem{Construction}[equation]{Construction}
\theoremstyle{remark}
\newtheorem{Remark}[equation]{Remark}
\newtheorem*{Remark*}{Remark}
\newtheorem{rmk}[equation]{Remark}
\newtheorem*{rmk*}{Remark}

\setlist{topsep=0pt,itemsep=6pt}

\def \PP {{\mathbb P}}
\def \N {{\mathbb N}}
\def \Z {{\mathbb Z}}
\def \Q {{\mathbb Q}}
\def \R {{\mathbb R}}
\def \C {{\mathbb C}}
\def \t {{\mathfrak t}}
\def \b {{\mathfrak b}}
\def \g {{\mathfrak g}}
\def \h {{\mathfrak h}}
\def \Cx {{\C^{\times}}}

\def \vell {\vec{\ell}\,}
\def \ssminus {\smallsetminus}

\def \frakX {{\mathfrak X}}
\def \frakL {{\mathfrak L}}

\def \calO {{\mathcal O}}
\def \bfone {\mathbf{1}}
\def \bfzero {\mathbf{0}}
\def \bft {{\mathbf{t}}}
\def \calN {{\mathcal N}}

\def \on {\operatorname}

\DeclareMathOperator \Lie {Lie}

\def \Gr {\on{Gr}}
\def \St {\on{St}}
\def \GL {\on{GL}}
\def \Span {\on{Span}}
\def \pr {\on{pr}}

\def \SL {\operatorname{SL}}
\def \SU {\operatorname{SU}}
\def \hol {{\operatorname{hol}}}
\def \graded {{\operatorname{graded}}}
\def \reg {{\operatorname{reg}}}
\def \Xreg {\frakX_\reg}

\def \inv {^{-1}}

\def \calS {{\mathcal S}}
\def \wt {\widetilde}
\def \wh {\widehat}

\begin{document}

\title{Bott canonical basis?}

\author{Yael Karshon and Jihyeon Jessie Yang}
\address{University of Toronto Mississauga}
\email{karshon@math.toronto.edu}
%
\email{jessie.yang@utoronto.ca}


\begin{abstract} 
Expanding an idea of Raoul Bott,
we propose a construction of canonical bases 
for unitary representations 
that comes from big torus actions on families of Bott-Samelson manifolds.
The construction depends only on the choices of a 
maximal torus, a Borel subgroup,
and a reduced expression for the longest element of the Weyl group.
It relies on a conjectural vanishing of higher cohomology 
of sheaves of holomorphic sections of certain line bundles
on the total spaces of the families, hence the question mark in the title. 
\end{abstract}



\maketitle

\setcounter{tocdepth}{1}
\tableofcontents

\section{Overview}
\labell{sec:overview}

Let $K$ be a compact connected Lie group.  
Fix a maximal torus $T$ in $K$.
We identify the Lie algebra of $S^1$ with $\R$
such that the exponential map becomes $t \mapsto e^{it}$.
Homomorphisms from $T$ to $S^1$ are determined by their differentials
at the identity; these differentials form the weight lattice\footnote{
Some authors reserve the term ``weight lattice''
to the case that $K$ is simply connected;
when studying representations of the Lie algebra, only this case matters.
}
$\t^*_\Z$
in the dual $\t^*$ of the Lie algebra $\t$ of $T$.
For a weight $\mu \in \t^*_\Z$,
we denote the corresponding homomorphism by $a \mapsto a^\mu$,
and we denote by $\C_\mu$ the representation $a \colon z \mapsto a^{\mu}z$
of $T$ on the vector space $\C$.

Let $V$ be an irreducible unitary representation of $K$.
For every weight $\lambda \in \t^*_\Z$, the corresponding weight space
consists of those vectors on which each torus element $a \in T$
acts as scalar multiplication by $a^\lambda$.
The space $V$ decomposes as a direct sum of the weight spaces.
If $K = \SU(2)$, the non-trivial weight spaces are one dimensional.
For other Lie groups, the weight spaces might be higher dimensional.
By a \textbf{canonical basis} we refer to a further splitting 
of the weight spaces
into one dimensional spaces that depends on as few choices as possible.\footnote
{ 
In the literature, \emph{canonical basis} is a basis of  a quantum version $U_q(\mathfrak{g})$ of a universal enveloping algebra.
This involves a parameter $q$.
The ordinary universal enveloping algebra corresponds to $q=1$.
\emph{Crystal bases} of representations of $U_q(\mathfrak{g})$ occur at $q=0$.
See \cite{hong-kang, kashiwara1, kashiwara2, kashiwara3, lusztig1, lusztig2}.  It would be interesting to explore analogies 
with one-parameter versions 
of our multi-parameter equivariant family. 
}

Let $G$ be the complexification of $K$;
it is a connected reductive\footnote{
If $K$ is simply connected, then $G$ is semisimple.
}
complex Lie group that contains $K$
and whose Lie algebra is the complexification of the Lie algebra of $K$.
Passing from $K$ to $G$ gives a bijection
from the set of compact connected Lie groups modulo isomorphism
to the set of connected reductive complex Lie groups modulo isomorphism.
Every irreducible unitary representation of the real group $K$
uniquely extends to an irreducible representation of the complex group $G$;
this gives a bijection from the set of irreducible unitary representations
of $K$ modulo isomorphism to the set of irreducible complex representations
of $G$ modulo isomorphism.

Let $\Delta \subset \t^*_\Z$ be the set of roots of $(K,T)$.
Consider the root space decomposition
$$\g = \h \oplus \bigoplus_{\alpha \in \Delta} \g_{\alpha} .$$
Here $\h$ is the Lie algebra of the Cartan subgroup $H$,
which, in turn, is the complexification of the maximal torus $T$.
Choose a Borel subgroup $B$ of $G$ that contains~$T$.
The choice of $B$ determines 
the set of positive roots $\Delta_+$ in $\Delta$,
the positive Weyl chamber $\t^*_+$ in $\t^*$,
the Bruhat (partial) order on $\t^*_\Z$, 
and the set $\{ \alpha_1, \ldots, \alpha_r \}$ 
of simple positive roots.\footnote{
The $\alpha_i$ span $\t^*$ if and only if $G$ is semisimple.
}
The Borel subgroup  $B$ is the connected subgroup of $G$
whose Lie algebra 
is $\h \oplus \bigoplus_{\alpha \in \Delta_+} \g_{\alpha}$.
The maximal unipotent subgroup of $B$ is the connected Lie subgroup $U$
whose Lie algebra is $\bigoplus_{\alpha \in \Delta_+} \g_{\alpha}$.
The subgroup $U$ is normal in $B$.
The multiplication map 
$(h,u) \mapsto hu$ identifies $B$ with the semi-direct product $H \ltimes U$,
so we have the projection with kernel $U$
$$
 \psi_H \colon B \to H \quad , \quad hu \mapsto h
$$
from the Borel subgroup $B$ to the Cartan subgroup $H$.
For each weight $\mu$ of $T$,
the corresponding homomorphism $a \mapsto a^\mu$ from $T$ to $S^1$ 
extends to a unique holomorphic homomorphism from $B$ to $\Cx$
that is trivial on $U$; we denote it by $b \mapsto b^\mu$.
Thus, $b^\mu = (\psi_H(b))^\mu$.
We let $B$ acts on $\C_\mu$ through this homomorphism.
Thus, for $b \in B$ and $z \in \C_\mu$, 
we write $b \colon z \mapsto b^\mu z$.

The Borel-Weil theorem gives a geometric construction
of irreducible representations of $K$
as spaces of holomorphic sections of $K$-equivariant holomorphic line bundles.
We now recall this construction.
Consider the holomorphic principal $B$ bundle $G \to G/B$
and the associated holomorphic line bundle
$L_\lambda = G \times_B \C_{-\lambda}$.
Namely, $L_\lambda$ is the quotient of $G \times \C_{-\lambda}$
by the anti-diagonal (right) $B$ action 
$b \colon (g,z) \mapsto (gb, b^\lambda z)$.
The space $G/B$, being a quotient of a complex manifold
by a free and proper holomorphic group action of a complex group,
is a complex manifold;
similarly, $L_\lambda$ is a holomorphic line bundle.
Left multiplication gives an action of $G$ on $L_\lambda$,
inducing a linear representation of $G$
on the vector space $\Gamma_\hol( G/B , L_\lambda )$ 
of holomorphic sections of $L_{\lambda}$.
By the Borel-Weil theorem,
the construction $\lambda \mapsto \Gamma_\hol( G/B , L_\lambda )$ 
gives a bijection from the set $\t^*_\Z \cap \t^*_+$ 
of \textbf{dominant weights}
to the set of irreducible $G$ representations modulo isomorphism,
under which the dominant weight $\lambda$ corresponds to the
(unique up to isomorphism) irreducible $G$ representation $V_\lambda$
with highest weight $\lambda$.

Raoul Bott had hoped to obtain a canonical basis
from a big torus action, by applying the following general argument.\footnote
{Raoul\ Bott, unpublished letter to M.\ F.\ Atiyah (1989). See the
introduction of \cite{grossberg:thesis}.} 
Let $M$ be a connected complex manifold, with a holomorphic action of a torus.
Let $L \to M$ be a holomorphic line bundle,
with a lifting of the action to a holomorphic action on $L$.
Then the space of holomorphic sections of $L$
decomposes into weight spaces for the torus action.
If $M$ is connected, the action has a fixed point,
and the dimension of the torus is equal to half the real dimension of $M$, 
then this gives a decomposition of the space of holomorphic sections
into one dimensional spaces.
The proof is as follows.\footnote{Yael Karshon
learned this argument from Michael Grossberg, who learned it from Raoul Bott.}
Fix a fixed point (no pun intended), $p$.
Let $\mu$ be the weight by which the torus acts on the fibre of $L$ 
above~$p$.  Fix a trivialization of the bundle $L$
over a neighbourhood of $p$,
a holomorphic local chart on a possible smaller neighbourhood of $p$,
and an identification of the torus with $(S^1)^n$,
such that near $p$ the torus action on $L$
becomes the action of $(S^1)^n$ on $D \times \C_\mu$
where $D$ is a neighbourhood of the origin in $\C^n$
and $(S^1)^n$ acts by coordinatewise multiplication on $D$
and by the weight $\mu$ on $\C_\mu$.
The restriction map to this neighbourhood,
composed with the coordinate chart and with the trivialization,
gives an equivariant linear map from the space of holomorphic sections
$M \to L$
to the space of holomorphic functions $D \to \C_\mu$.
By analytic continuation, this restriction map is one to one.
But on the space of holomorphic functions $D \to \C_\mu$,
every weight space is spanned by a monomial, hence is one dimensional.

In the Borel-Weil setup,
the dimension of the torus $T$ is generally less than half 
the real dimension of the base manifold $G/B$,
which is not big enough for applying the above argument
to the left multiplication action of $T$ on the holomorphic line bundle
$L_\lambda = G \times_B \C_{-\lambda}$.
But if we only want to keep track of the restrictions to $T$ (or to $B$)
of irreducible representations of~$G$,
it turns out \cite[\S 5.1]{demazure} that 
instead of working with the flag manifold $G/B$, we can work with the 
Bott-Samelson manifold, which we will now describe.
And the Bott-Samelson manifold
does admit an action of a torus of the right dimension.
Bott had hoped to use this torus action on the Bott-Samelson manifold
to obtain a canonical basis.

The Bott-Samelson manifold, introduced by Raoul Bott and Hans Samelson 
in \cite{bott-samelson-1,bott-samelson-2}, can be constructed 
in the following way.  
Let $\{\alpha_1,\ldots,\alpha_r\}$ be the set of simple positive roots.
For each simple positive root $\alpha_i$, let $P_{\alpha_i}$ be 
the corresponding minimal parabolic subgroup of $G$.
The intersection $K_{\alpha_i} := K \cap P_{\alpha_i}$
is a maximal compact subgroup of $P_{\alpha_i}$,
and the quotient $K_{\alpha_i}/T$ is diffeomorphic to a 2-sphere.
Let $T_{\alpha_i}$ be the centre of $K_{\alpha_i}$;
it is a codimension one subtorus of $T$.
Let $\alpha_{i_1}, \ldots, \alpha_{i_n}$ be a finite sequence
of simple positive roots. 
The corresponding Bott-Samelson manifold, 
$$ Y_{\alpha_{i_1},\ldots,\alpha_{i_n}} :=
K_{\alpha_{i_1}} \times_T K_{\alpha_{i_2}} \times_T \cdots 
                   \times_T K_{\alpha_{i_n}}/T ,$$
is the quotient of $K_{\alpha_{i_1}} \times \cdots \times K_{\alpha_{i_n}}$
by the right\footnote{For an abelian group, a left action
is the same thing as a right action.
We write this action as a right action because of its relation
with the action of $B \times \cdots \times B$ that we describe below.} 
action of $T \times \cdots \times T$ that is given by
$ (k_1, k_2, \ldots, k_n) \cdot (a_1,\ldots,a_n) 
 = (k_1 a_1, a_1\inv k_2 a_2, \ldots , a_{n-1}\inv k_n a_n)$.
Successively truncating the last factors, we obtain a tower of bundles
\begin{equation*} 
 Y_{\alpha_{i_1},\ldots,\alpha_{i_n}}
 \to Y_{\alpha_{i_1},\ldots,\alpha_{i_{n-1}}}
 \to \ldots \to  Y_{\alpha_{i_1},\alpha_{i_2}} \to Y_{\alpha_{i_1}}
\end{equation*}
whose fibres are two-spheres.
We have a left action of $T \times \cdots \times T$ 
on the Bott-Samelson manifold,
given by
$(a_1,\ldots,a_n) \cdot [k_1,\ldots,k_n] = [a_1 k_1,\ldots,a_n k_n]$.
This action is not faithful.  
But it descends to a faithful action of the \textbf{Bott-Samelson torus},
$ T \times_{T_{\alpha_{i_1}}} T \times_{T_{\alpha_{i_2}} } \cdots 
  \times_{T_{\alpha_{i_{n-1}}} } T / T_{\alpha_{i_n}}$,
whose dimension, $n$, is equal to half the dimension 
of the Bott-Samelson manifold.
The multiplication map $[k_1,\ldots,k_n] \mapsto k_1 \cdots k_n T$
is a $T$ equivariant smooth map 
from the Bott-Samelson manifold 
$Y_{\alpha_{i_1},\dots, \alpha_{i_n}}$ to~$K/T$,
which has degree one 
if the sequence of roots $\alpha_{i_1}, \ldots , \alpha_{i_n}$
corresponds to a reduced expression of the longest element 
of the Weyl group~\cite{bott-samelson-1}.

The relevance of the Bott-Samelson manifold to representation theory
is obtained from a complex geometric construction of this manifold,
which  is due to Hansen \cite{hansen} and Demazure~\cite{demazure};
also see Jantzen \cite[Ch.~13--14]{jantzen}.
We now describe this complex Bott-Samelson manifold. 
We begin by recalling that the manifold and line bundle
of the Borel-Weil theorem have both a compact construction
and a complex construction.  The complex construction
consists of the manifold $G/B$
and the line bundle 
$$ L_\lambda := G \times_B \C_{-\lambda},$$
which we already described.
The compact construction consists of the manifold $K/T$
and the line bundle 
$$ L_\lambda^K := K \times_T \C_{-\lambda},$$
which is the quotient of $K \times \C_{-\lambda}$
by the $T$ action $a \colon (g,z) \mapsto (ga,a^\lambda z)$.
The group $K$ acts on $L_\lambda^K$ by left multiplication.
The inclusion map $K \to G$
induces a diffeomorphism $K/T \to G/B$
and an isomorphism of $K$-equivariant complex line bundles
$L_\lambda^K \to L_\lambda$.
Similarly, we have both a compact construction and a complex construction
for the Bott-Samelson manifolds.  We already described the compact construction.
For the complex construction, take
\begin{equation} \labell{Zalpha}
 Z_{\alpha_{i_1} , \ldots , \alpha_{i_n}}
 := P_{\alpha_{i_1}} \times_B P_{\alpha_{i_2}} \times_B \cdots 
                      \times_B P_{\alpha_{i_n}}/B ,
\end{equation}
with the holomorphic line bundle
\begin{equation} \labell{L lambda Z}
 L^Z_\lambda := P_{\alpha_{i_1}} \times_B P_{\alpha_{i_2}} \times_B \cdots 
                               \times_B P_{\alpha_{i_n}}
\times_B \C_{-\lambda},
\end{equation} 
obtained as the quotient of 
$P_{\alpha_{i_1}} \times \cdots \times P_{\alpha_{i_n}} \times \C_{-\lambda}$
by the right action of $B \times \cdots \times B$ that is given by
$ (p_1, \, p_2, \, \ldots, \, p_n; z) \cdot (b_1, \, \ldots, \, b_n)  
 = (p_1 b_1, \, b_1\inv p_2 b_2, \, \ldots , \, b_{n-1}\inv p_n b_n ; 
    b_n^\lambda z) \, $.
The Borel subgroup $B$ acts by left multiplication on the first factor.
These fit into the following commuting diagram
(which we don't bother completing to a cube)
in which the squares are pullback diagrams,
the vertical arrows are complex line bundles,
the front square consists of $B$-equivariant holomorphic maps,
the arrows that point from the back to the front
are $T$-equivariant diffeomorphisms induced from inclusion maps,
and the arrows that point from the left to the right
are induced from multiplication maps.
$$ \xymatrix{
 & & & L_\lambda^K \ar[ld] \ar[dd] \\
 L_\lambda^Z \ar[rr] \ar[dd] && L_\lambda \ar[dd] & \\ 
 & Y_{\alpha_{i_1},\ldots,\alpha_{i_n}} \ar'[r][rr] \ar[dl] && K/T \ar[ld] \\
 Z_{\alpha_{i_1},\ldots,\alpha_{i_n}} \ar[rr] && G/B
}$$
By pulling back, 
the front square gives a $B$ equivariant map
from the space of holomorphic sections of $L_\lambda$
to the space of holomorphic sections of $L^Z_\lambda$.
If the sequence of roots $\alpha_{i_1}, \ldots , \alpha_{i_n}$
corresponds to a reduced expression of the longest element 
of the Weyl group, then this map
is an isomorphism of $B$ representations;
the space of holomorphic sections of $L^Z_\lambda$
then provides a model for the irreducible representation $V_\lambda$
with highest weight $\lambda$,
with the $G$ action restricted to a $B$ action.
See \cite{demazure,hansen,jantzen}.

Raoul Bott had hoped to use the Bott-Samelson torus action,
combined with the realization of $V_\lambda$
as the space of holomorphic sections of a line bundle 
over the Bott-Samelson manifold, 
to obtain a canonical basis for~$V_\lambda$.
Unfortunately, this beautiful idea did not quite work.
This is because the action of the Bott-Samelson torus
$T \times_{T_{\alpha_{i_1}}} \times_{T_{\alpha_{i_2}}} T \cdots
 \times_{T_{\alpha_{i_{n-1}}}} T / T_{\alpha_{i_n}}$
on the Bott-Samelson manifold is not holomorphic.
To make this action holomorphic,
Michael Grossberg deformed the complex structure.
The resulting toric variety, which Grossberg called a \emph{Bott tower}
\cite{grossberg:thesis,grossberg-karshon}, is the quotient 
\begin{equation} \labell{Xalpha}
 X_{\alpha_{i_1},\ldots,\alpha_{i_n}} 
 := (P_{\alpha_{i_1}} \times \cdots \times P_{\alpha_{i_n}} ) / B^n 
\end{equation}
by the right $B^n$ action that is given by
$$
 (p_1,p_2,\ldots,p_n) \cdot (b_1,\ldots,b_n) = 
(p_1 b_1 \, , \, \psi_H(b_1)\inv p_2 b_2 \, , \, \ldots \, , \, 
   \psi_H(b_{n-1})\inv p_n b_n),
$$
where $\psi_H \colon B \to H$ is the projection map
from the Borel to the Cartan.
Over the Bott tower we have the holomorphic line bundle
\begin{equation} \labell{L lambda X} 
 L_\lambda^X := (P_{\alpha_{i_1}} \times \cdots 
   \times P_{\alpha_{i_n}} ) \times_{B^n} \C_{-\lambda} ,
\end{equation}
which is the quotient of 
$P_{\alpha_{i_1}} \times \cdots \times P_{\alpha_{i_n}} \times \C_{-\lambda}$
by the right action of $B \times \cdots \times B$ that is given by
$(p_1,p_2,\ldots,p_n,z) \cdot (b_1,b_2,\ldots,b_n) = 
 ( p_1 b_1, \psi_H(b_1)^{-1} p_2 b_2, \ldots, 
   \psi_H(b_{n-1})^{-1} p_n b_n, b_n^\lambda z )$. 

On the Bott tower $X_{\alpha_{i_1},\ldots,\alpha_{i_n}}$,
we have a left action of $T \times \cdots \times T$ given by 
\begin{equation} \labell{left action on bott tower}
(a_1,\, \ldots,\, a_n) \cdot [p_1,\, \ldots,\, p_n] 
 = [a_1 p_1,\, \ldots,\, a_n p_n],
\end{equation}
which is well defined and holomorphic,
and which descends to a holomorphic action of the Bott-Samelson torus
$T \times_{T_{\alpha_{i_1}}} \cdots \times_{T_{\alpha_{i_{n-1}}}} 
   T / T_{\alpha_{i_n}}$.
(In contrast, on the complex Bott Samleson manifold \eqref{Zalpha},
the formula~\eqref{left action on bott tower}
does not give a well defined action.)
This action lifts to the line bundle $L_\lambda^X$ by 
\begin{equation} \labell{left action on LX}
 (a_1,\, \ldots,\, a_n) \cdot [p_1,\, \ldots,\, p_n,\, z] 
  = [a_1 p_1,\, \ldots,\, a_n p_n,\, z] \,.
\end{equation}
The dimensions now are as in Bott's earlier argument.
Unfortunately, again this doesn't work.
The space of holomorphic sections over the Bott tower
does decompose into one dimensional weight spaces,
but this space of holomorphic sections
can no longer be used as a model for the representation~$V_\lambda$.
The Bott tower and the complex Bott-Samelson manifold
give two different complex structures
on the real Bott-Samelson manifold $Y_{\alpha_{i_1},\ldots,\alpha_{i_n}}$,
and these complex structures yield different spaces of holomorphic sections.

Michael Grossberg, in his PhD thesis \cite{grossberg:thesis},
carried out a thorough analysis of this situation;
his work, together with a presymplectic analogue,
later appeared in his joint paper~\cite{grossberg-karshon} with Karshon.
Specifically, Michael Grossberg obtained a one-parameter family
of complex structures on the Bott-Samelson manifold
by deforming the right action of $B \times \ldots \times B$
by which we take the quotient.
His family depended on a parameter in $[0,\infty]$,
with the parameter value~$0$ giving the complex Bott-Samelson manifold
and the parameter value $\infty$ giving the Bott tower. 

In connection with this work, around 1993, Joseph Bernstein 
suggested to Yael Karshon that in this situation --- a deformation 
of a complex manifold into one that admits a toric action --- it is 
worthwhile to try to fit the deformation
into a family of complex manifolds,
with the torus acting on the entire family. 
Torus elements generally take each fibre of the family 
to a different fibre of the family,
but over a fixed point in the base
we get a \emph{special fibre} on which the torus acts.

Pasquier, in~\cite{pasquier}, provided an algebraic geometric interpretation
of the work of Michael Grossberg.
He constructed a one-parameter family of varieties,
parametrized by $\C$,
such that generic fibres are Bott-Samelson manifolds
and the fibre over the origin $0 \in \C$ is the Bott tower.

Following Bernstein's idea, we sought --- and eventually found --- 
a torus-equivariant family
where the complex Bott-Samelson manifold occurs as a generic fibre 
and where the toric Bott tower occurs as a special fibre over a fixed point.
We describe this \emph{Bott-Samelson family} 
in Section~\ref{sec:Bott-Samelson}.
Our construction is an extension of Grossberg's deformation
into a deformation that depends on $n$ parameters.
We work in the complex analytic (rather than algebraic) setup. 
Our Bott-Samelson family is a family of complex manifolds,
parametrized by $\C^n$ where $n$ is the complex dimension of the 
Bott-Samelson manifold.  
The complex torus $(\Cx)^n$ acts holomorphically on the entire family, 
making the special fibre over the origin $0 \in \C^n$ into a toric variety.

In Section~\ref{sec:filtrations} we give a general setup
for how a family of this type
can give rise to a splitting into lines
of the space of holomorphic sections over the generic fibre.
In Section~\ref{sec:Bott-Samelson} we show how Bott-Samelson manifolds fit into this general setup.  In Section~\ref{sec:canonical-basis} we show that 
together with the results of Section~\ref{sec:filtrations},
under a conjectural ``vanishing of higher cohomology''
(which we hope is true!),
the construction in   Section~\ref{sec:Bott-Samelson} gives canonical bases (in an appropriate sense)
for irreducible representations of compact Lie groups.
In Section~\ref{sec:examples} we work out an example.
In this example, the elements of our basis
are indexed by a collection of lattice points
whose convex hull is upper-triangularly and unimodularly equivalent
to the corresponding string polytope;
see Remark~\ref{string polytope}.
In Appendices~\ref{app:hadamard} and~\ref{app:indep}
we prove results that we need 
for our proofs in Section~\ref{sec:filtrations}.
Specifically,
in Appendix~\ref{app:hadamard} we develop a ``holomorphic Hadamard Lemma''
 --- see Theorem~\ref{HA} -- which we find interesting in its own right.

\begin{figure}[ht]
{\centering}
\includegraphics[scale=0.6]{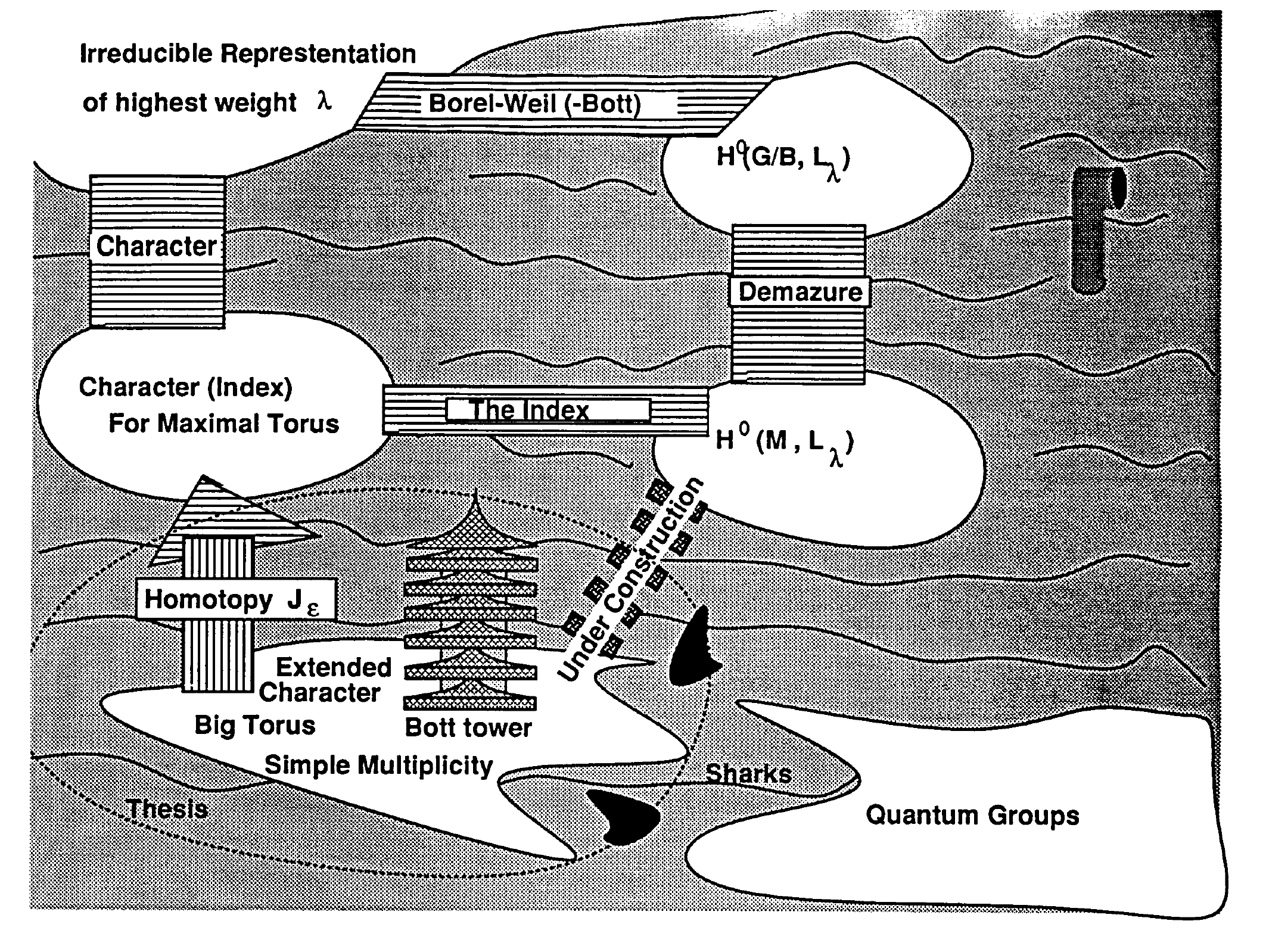}
\captionsetup{labelformat=empty}
\caption[Excerpt from Michael Grossberg's Ph.D.\ thesis]
{Excerpt from Michael Grossberg's Ph.D.\ thesis
\cite[p.~46]{grossberg:thesis},\footnotemark}
reproduced with permission.
\end{figure}
\footnotetext{Copyright 1991, Michael David Grossberg}

\section{Filtrations induced by equivariant families}
\labell{sec:filtrations}
In this section, we consider a complex analytic manifold $X$
equipped with a holomorphic line bundle $L$
that occurs as a generic fibre of an $(\Cx)^n$-equivariant family 
(defined below) $\frakL \to \frakX \to \C^n$.
Under an additional mild technical condition
(that is true in the setup of our Sections~\ref{sec:Bott-Samelson}
and~\ref{sec:canonical-basis}),
we use the family to obtain a filtration $(F_{\vell})$ of the space 
$$V := \Gamma_\hol(X,L)$$ 
of holomorphic sections of $L$ over $X$ that is indexed by the set $\Z^n$ 
with its product partial ordering: 
for $\vell=(\ell_1,\dots, \ell_n)$
and $\vell'=(\ell_1', \dots, \ell_n')$,
\begin{equation} \labell{product partial ordering}
    \vell\le \vell' \quad \text{ iff } \quad \ell_j\le \ell'_j
 \ \text{ for all } \ j=1, \dots, n  . 
\end{equation}

Assuming the vanishing of higher cohomology of the sheaf of holomorphic
sections of $\frakL \to \frakX$, we show that $V$ is isomorphic 
to the associated graded space $\bigoplus F_{\vell}/F_{>\vell}$.
If the special fibre of the family is toric, 
the ``leaves''\footnote{
We follow~\cite{kaveh-khovanskii}
in calling the graded pieces $F_{\vell}/F_{>\vell}$ ``leaves''.
}
 $F_{\vell}/F_{>\vell}$ are one dimensional,
and an identification of $V$ with the associated graded space
yields a decomposition of $V$ into one dimensional subspaces.
In the setup of our Sections~\ref{sec:Bott-Samelson}
and~\ref{sec:canonical-basis}, this identification can be made canonical,
namely, without any additional choices.

\subsection*{Equivariant families over $\bm{\C^n}$}\ 

In this paper, an \textbf{equivariant family over $\C^n$}
refers to the following data.

\begin{enumerate}[noitemsep,topsep=0pt]
\item
A connected complex manifold $\frakX$, with a holomorphic $(\Cx)^n$-action.
\item A $(\mathbb{C}^{\times})^n$-equivariant holomorphic proper submersion
$\pi\colon \frakX \to \C^n$,
where $(\mathbb{C}^{\times})^n$ acts on $\mathbb{C}^n$ by
coordinatewise multiplication.
\item
A $(\Cx)^n$-equivariant holomorphic line bundle $\frakL$ over $\frakX$.
\end{enumerate}

An \textbf{action} of a Lie group $H$ on the equivariant family 
is a holomorphic\footnote{for a real Lie group, we require each element
of $H$ to act holomorphically; for a complex Lie group,
we require the action map $H \times \frakL \to \frakL$ to be holomorphic.}
action of $H$ on $\frakX$ and $\frakL$
that commutes with the $(\Cx)^n$ action
and such that the maps $\frakL \to \frakX \to \C^n$ are $H$ equivariant,
where $H$ acts on $\C^n$ trivially.

Let $\bfzero := (0,\ldots,0)$ be the origin in $\C^n$ 
and $X_\bfzero := \pi^{-1}( \{ \bfzero \} )$ the special fibre in $\frakX$.
Then $(\Cx)^n$ acts on $X_\bfzero$.
Later we will assume the following additional 
\textbf{mild technical condition}:
\begin{equation} \labell{mild technical}
\begin{minipage}{4.5in}
For every $j \in \{ 1,\ldots,n \}$,
the action of the $j$th factor $(\Cx)_{j^{\text{th}}}$ of $(\Cx)^n$
on the special fibre $X_{\bfzero}$ 
has a fixed point at which all the isotropy weights are non-positive.
\end{minipage}
\end{equation}

The following remark is a consequence of Ehresmann's lemma.
We use it in the proof of Proposition \ref{prop:1}.

\begin{rmk}\labell{rk:connected}
Given an equivariant family $\frakL \to \frakX \to \C^n$,
there exists an $(S^1)^n$-equivariant diffeomorphism
of $\frakX$ with $X_{\bfzero} \times \C^n$
that takes $X_{\bft} := \pi^{-1}(\{\bft\})$ 
onto $X_{\bfzero} \times \{ \bft \}$,
which gives a family of complex structures on $X_\bfzero$ 
that depend smoothly on the parameter~$\bft$.
This has the following consequences.
\begin{enumerate}
\item[(a)]
(Since $\frakX$ is connected,) for every $\bft \in \C^n$
the fibre $X_\bft$ is connected,
and for every $j \in \{ 1, \ldots, n \}$ the preimage 
$\pi^{-1}(\Cx \times \ldots \times \Cx \times \{ 0 \}_{j^{\text{th}}} 
 \times \Cx \times \ldots \times \Cx)$ is connected.
\item[(b)]
For any $j \in \{1,\ldots,n\}$ and $\bft \in 
\C \times \cdots \times \C \times \{0\}_{j^{\text{th}}} 
\times \C \times \cdots \times \C$,
the mild technical condition holds for $X_\bfzero$
if and only if the analogous condition holds for $X_\bft$:

\begin{quotation}
The action of the $j^{\text{th}}$ factor 
$(\Cx)_{j^{\text{th}}}$ of $(\Cx)^n$
on the fibre $X_\bft$ 
has a fixed point at which all the isotropy weights are non-positive.
\end{quotation}

\end{enumerate}
\end{rmk}

\begin{rmk} \labell{technical condition projective}
The mild technical condition~\eqref{mild technical}
holds whenever $X_\bfzero$ is compact and K\"ahler
and the $(S^1)^n$ action is Hamiltonian.
Indeed, as a consequence of the local normal form
for Hamiltonian torus actions,
a maximum point for the $j$th component of the momentum map
(or a minimum point, depending on the sign conventions)
is a fixed point (for $(S^1)_{\text{$j$th }}$, hence) for $(\Cx)_\text{$j$th}$
with non-positive isotropy weights.
In particular, Condition~\eqref{mild technical} holds 
if there is a projective embedding of $X_\bfzero$
such that the $(S^1)^n$ action is induced
from an action on the ambient projective space.
\end{rmk}

\begin{rmk}
In the case of $n=1$ and in an algebraic setting,
the notion of an equivariant family over $\C^n$
becomes the notion of a \emph{test configuration},
which was introduced by Donaldson in 2002~\cite{donaldson}.
The idea to use an equivariant family
to obtain a filtration of the space of holomorphic sections 
$\Gamma_\hol(X,L)$ of the line bundle $L$ over $X$
is also inspired by the work of Witt Nystrom
with test configurations in \cite{witt-nystrom}.
\end{rmk} 

\subsection*{The vanishing of higher cohomology assumption}\ 

We will describe how to use an equivariant family 
$\frakL \to \frakX \to \C^n$
to obtain a filtration of the space of holomorphic sections
$V := \Gamma_{\hol}( X_\bfone , \frakL|_{X_\bfone} )$ where
$$ X_{\bfone} := \pi\inv(\bfone) \, 
   \quad \text{ with } \quad \bfone := (1,\ldots,1).$$
The filtration $\{ F_{\vell} \}$ will be indexed by $n$-tuples of integers
with the product partial ordering ~\eqref{product partial ordering}.
In order to identify the associated graded space
$\oplus_{\vell \in \Z^n} \, F_{\vell}/F_{>\vell}$\,  with $V$,
we will need to assume the \textbf{vanishing of higher cohomology}
of the sheaf $\calO_\frakL$ of holomorphic sections of $\frakL$ over~$\frakX$:
$$ H^{\geq 1} (\frakX,\calO_\frakL) = \{ 0 \} .$$

The case that interests us is that of Bott-Samelson families
with line bundles that come from dominant weights.
We do not know if the vanishing of higher cohomology holds in this case;
we hope that it does.

\subsection*{``Sweep, twist, extend''}\ 

Fix an equivariant family $\frakL \to \frakX \to \C^n$
with an $H$-action.
The $(\Cx)^n$ action on the family
induces a $(\Cx)^n$ action on the space of holomorphic sections,
as well as on the space of holomorphic sections
over any $(\Cx)^n$-invariant open subset of $\frakX$.

Consider the space 
$ \Gamma_\hol 
   \left( \frakX_\reg , \frakL|_{\frakX_\reg} \right) ^{ (\Cx)^n } $
of $(\Cx)^n$-invariant holomorphic sections of $\frakL$ 
over the open dense subset 
\[ \Xreg := \pi^{-1} ( (\Cx)^n ) \quad \] 
of $\frakX$.
The restriction map
$$ \Gamma_\hol( \frakX_\reg , \frakL|_{\frakX_\reg} ) ^{ (\Cx)^n }
 \, \to \, \Gamma_\hol \left( X_\bfone , \frakL|_{X_\bfone} \right)
\ , \qquad \sigma \mapsto \sigma|_{X_{\bfone}} $$
is an ($H$-equivariant) linear isomorphism,
with inverse given by 
$$ \Gamma_\hol \left( X_\bfone , \frakL|_{X_\bfone} \right)
 \, \to \, 
   \Gamma_\hol( \frakX_\reg , \frakL|_{\frakX_\reg} ) ^{ (\Cx)^n }
\ , \qquad s \mapsto \tilde{s},  $$
where
$$\tilde{s}(x) := \tau \cdot s( \tau^{-1} \cdot x)
\qquad \text{ for } \quad \tau = \pi(x).$$
We say that $\tilde{s}$ is the \textbf{sweep} of $s$ by the $(\Cx)^n$-action.

The restriction map 
$$ \Gamma_\hol \left( \frakX , \frakL \right)
   \to \Gamma_\hol \left( \frakX_\reg , \frakL|_{\frakX_\reg} \right)$$
is $H \times (\Cx)^n $-equivariant and one-to-one (but not onto;
see Proposition~\ref{p:lsj}).
Denote its inverse 
(which is defined on the image of the restriction map) by 
$$ \sigma \mapsto \ol{\sigma} .$$
Thus, $\ol{\sigma}$ is the (unique continuous) extension 
of $\sigma$ to $\frakX$.
The graph of $\ol{\sigma}$ is the closure in $\frakL$ of the graph of $\sigma$.

For any lattice vector $\vec{\ell} = (\ell_1,\ldots,\ell_n)$ in $\Z^n$, 
we denote by $\mathbf{t}^{-\vell}$ the meromorphic function on~$\frakX$ 
that is given by 
\[
 \mathbf{t}^{-\vell}(x) := t_1^{-\ell_1}\cdots t_n^{-\ell_n}
 \quad \text{ when } \ \pi(x)=(t_1,\dots, t_n) .
\]
The corresponding \textbf{twist} of a holomorphic section $\sigma$ 
over $\frakX_\reg$
is the section $\bft^{-\vell} \sigma$ over $\frakX_\reg$.

Given a section $s \in \Gamma_{\hol} (X_\bfone, ,\frakL|_{X_{\bfone}})$,
its sweep $\tilde{s}$ might not extend to $\frakX$,
but in Propositions~\ref{prop:all j} and \ref{prop:1} we show that 
after an appropriate twist we obtain a section that does extend to~$\frakX$.

\subsection*{The filtration}\ 

As before, fix an equivariant family $\frakL \to \frakX \to \C^n$
with an $H$-action.
For each  $\vell \in \Z^n$, we consider the ($H$ invariant) space 
of those sections whose ``sweep and twist'' extends to $\frakX$:
\[ 
   F_{\vec{\ell}} := 
   \{ s \in \Gamma_\hol(X_\bfone,\frakL|_{X_\bfone}) \ | \ \,
 \mathbf{t}^{-\vec{\ell}}\tilde{s}
\text{ extends to a holomorphic section of $\frakL \to \frakX$ } \}.
\]
The spaces $F_{\vec{\ell}}$ form a decreasing filtration 
of  $\Gamma_\hol(X_\bfone,\frakL|_{X_\bfone})$ 
with respect to the product partial ordering~\eqref{product partial ordering},
in the following sense:
$$ \text{ If } \ \vell \leq \vell',  
   \ \text{ then } \ F_{\vell} \supseteq F_{\vell'}.$$
(Indeed, let $s \in F_{\vell'}$, and let $\vell \leq \vell'$.
Then $\bft^{-\vell'} \tilde{s}$ extends to $\frakX$,
and $\bft^{\vell'-\vell}$ is holomorphic on $\frakX$.
So their product $\bft^{-\vell} \tilde{s}$ extends to $\frakX$,
and so $s \in F_{\vell}$.)  
In particular,
\[ 
    F_{>\vell} \, := \,
    \text{span} \bigcup_{\vell'>\vell} F_{\vell'} \] 
is a subspace of $F_{\vell}$.

\begin{Remark}
From the $(\Cx)^n$ action, we get the weight spaces\footnote
{
The action of an element $\tau \in (\Cx)^n$ 
on a section $\sigma \colon \frakX \to \frakL$ 
is $(\tau \cdot \sigma)(x) := \tau \cdot (\sigma(\tau^{-1} \cdot x ))$,
where on the right hand side $\tau^{-1}$ acts on $\frakX$
and $\tau$ acts on $\frakL$.
The $\vell$ weight space for the $(\Cx)^n$ action on the space of sections 
is then
$ \Gamma_\hol(\frakX,\frakL)_{\vell} := 
 \{ \sigma \in \Gamma_\hol(\frakX,\frakL) \ | \ 
      \tau \cdot \sigma = \tau^{\vell}\sigma
\text{ for all } \tau \in (\Cx)^n \}$;
here $\tau \cdot \sigma$ refers to the action on the sections,
and $\tau^{\vell}\sigma$ refers to the multiplication 
by the scalar $\tau^{\vell}$.
}
\[
\Gamma_\hol(\frakX,\frakL)_{\vell} \quad , \quad \text{ for }
   \vell = (\ell_1,\ldots,\ell_n) \in \Z^n \, .
\]
The image of the restriction map
$ \Gamma_\hol \left( \frakX , \frakL \right)_{\vell}
\ \to \ \Gamma_\hol \left( X_\bfone, \frakL|_{X_\bfone} \right) $
is the space $F_{\vell}$.
The inverse of this restriction map is the ``sweep, twist, extend'' map,
$$ F_{\vell} \xrightarrow{\cong} 
   \Gamma_\hol \left( \frakX , \frakL \right)_{\vell}
 \, , \quad s \mapsto \ol{ t^{-\vell} \tilde{s} } \, ,$$
which is an $H$-equivariant linear isomorphism.
\end{Remark}

\subsection*{The associated graded space} \ 

\begin{Theorem} \labell{t:main}
Fix an equivariant family $ \frakL \to \frakX \to \C^n $ 
with an $H$-action.  
Let
$$  V := \Gamma_\hol(X_\bfone,\frakL|_{X_\bfone}) .$$
For each $\vell \in \Z^n$,
consider the ``sweep, twist, extend, restrict'' map
\begin{equation} \labell{sweep-twist-extend-restrict}
F_{\vell} 
   \to \Gamma_\hol \left( X_\bfzero , \frakL|_{X_\bfzero} \right) 
\quad , \quad 
    s \mapsto {\ol{ \bft^{-\vell} \tilde{s} }}|_{X_\bfzero} .
\end{equation}
Let
$$ V_\graded := \bigoplus\limits_{\vell} F_{\vell}/F_{>\vell} .$$
\begin{enumerate}
\item[(a)]
For each $\vell \in \Z^n$, the kernel of the 
``sweep, twist, extend, restrict'' map~\eqref{sweep-twist-extend-restrict}
contains $F_{>\vell}$,
and it is equal to $F_{>\vell}$
if the higher cohomology of the sheaf of holomorphic sections of $\frakL$
vanishes.
Thus, the ``sweep, twist, extend, restrict'' maps fit together into a map
$$ V_\graded \to \Gamma_{\hol}(X_\bfzero,\frakL|_{X_\bfzero}) ,$$
and 
if the higher cohomology of the sheaf of holomorphic sections of $\frakL$
vanishes, then the restriction of this map 
to each summand of $V_\graded$ is one-to-one.

\end{enumerate}

\medskip
For each $\vell \in \Z^n$, 
let $V_{\vell}$ be an $H$-invariant complement 
of $F_{>\vell}$ in $F_{\vell}$.
The composition
$V_{\vell} \xrightarrow{\text{inclusion}}
   F_{\vell} \xrightarrow{\text{quotient}}
   F_{\vell} / F_{>\vell}$
is an isomorphism;
let
$ i_{\vell} \colon F_{\vell}/F_{>\vell} \hookrightarrow V$ be 
the inverse of this isomorphism,
followed by the inclusion map $ V_{\vell} \xrightarrow{\text{inclusion}} V $.
\begin{enumerate}
\item[(b)]
Consider the ($H$-equivariant) map 
$(\sum_{\vell} i_{\vell}) \colon V_\graded \to V$.
If the mild technical condition~\eqref{mild technical} holds,
then this map is onto.
If the higher cohomology 
of the sheaf of holomorphic sections of $\frakL$ vanishes,
then this map is one-to-one.

\end{enumerate}
\end{Theorem}

Propositions~\ref{map to centre}, \ref{p:onto}, and~\ref{p:one-to-one}
below constitute the proof of Theorem~\ref{t:main}.

\begin{rmk}
In the setup of Theorem~\ref{t:main},
assuming the mild technical condition~\eqref{mild technical}
and the vanishing of higher cohomology,
if $V$ is equipped with a Hermitian inner product and $H$ acts unitarily
(or if $H$ is the complexification of a real Lie group
that acts unitarily),
then we can take $V_{\vell}$ to be the orthocomplement of $F_{>\vell}$.
\end{rmk}

Theorem~\ref{t:main} has the following corollary:

\begin{Corollary}
In the setup of Theorem~\ref{t:main},
we obtain an $H$-equivariant linear isomorphism
$$ V \to \bigoplus\limits_{\vell} F_{\vell} / F_{>\vell} 
 \ \left( = V_\graded \right), $$
and, for each $\vell$, an $H$-equivariant linear injection
of the ``leaf'' $F_{\vell}/F_{>\vell}$
into the space of sections over the special fibre:
$$ F_{\vell} / F_{>\vell} \hookrightarrow 
   \Gamma_{\hol} (X_\bfzero , \frakL|_{X_{\bfzero}} ) .$$

The image of this embedding is contained in the $\vell$ weight space of 
   $\Gamma_{\hol} (X_\bfzero , \frakL|_{X_{\bfzero}} ) $
as a representation of $(\Cx)^n$.
Thus, if this representation is multiplicity-free
(which means that its non-trivial weight spaces are one dimensional),
then we obtain a decomposition of $V$ into one dimensional 
$H$-invariant subspaces. 

More generally, the isotypic decompositions of the ``leaves''
$F_{\vell}/F_{>\vell}$ as representations of $H$
give a decomposition of $V$ into $H$ invariant subspaces
that refines the isotypic decomposition of $V$ as a representation of $H$.
\end{Corollary}

\subsection*{The sweep of every section can be extended after twisting } \

In this subsection, (under our mild technical condition,)
we show that, for every section $s\in \Gamma_\hol(X,L)$,
its ``sweep'' $\tilde{s}$ can be ``twisted''
such that it will extend to a holomorphic section of the line bundle $\frakL$
over the entire family $\frakX$.

We start by showing that we can consider one coordinate at a time.
We use this result in the criterion~\eqref{eq:Fl},
which we use in the proof of Proposition~\ref{p:onto}.

\begin{prop} \labell{prop:all j} 
Given any section $s\in \Gamma_\hol(X,L)$
and $\vec{\ell} = (\ell_1,\ldots,\ell_n) \in \Z^n$,
the section $\bft^{-\vell} \tilde{s}$ of $\frakL$ over $\frakX_\reg$
extends to a holomorphic section of $\frakL$ over $\frakX$
if and only if, for every $j \in \{1,\ldots,n\}$,
the section $t_j^{-\ell_j}\tilde{s}$
of $\frakL$ over $\frakX_{\reg}$
extends to a holomorphic section of $\frakL$ over
$\pi^{-1}(\C^{\times}\times\cdots\times\C^{\times}\times\
\C_{\text{$j$th}}\times\C^{\times}\times\cdots\times\C^{\times})$.
\end{prop}

\begin{proof}
By Hartog's lemma, $\bft^{-\vell} \tilde{s}$
extends to a holomorphic section over $\frakX$
if and only if, for every $j = 1,\ldots,n$,
it extends to a holomorphic section over the open set
$\pi^{-1}(
\Cx \times \ldots \times \Cx \times \C_{\text{$j$th}} \times \Cx
 \times \ldots \times \Cx)$.
This holds if and only if the section $t_j^{-\ell_j} \tilde{s}$
extends to a holomorphic section over this open set,
because the sections $\bft^{-\vec{\ell}} \tilde{s}$
and $t_j^{-\ell_j} \tilde{s}$
differ by multiplication by the monomial
$t_1^{\ell_1}\cdots \widehat{t_j^{\ell_j}} \cdots t_n^{\ell_n}$,
which is non-vanishing 
on $\pi^{-1}(\Cx \times \cdots \times \Cx \times \C_{j^{\text{th}}} 
\times \Cx \times \cdots \times \Cx)$.
\end{proof}

We now show that the ``sweep'' of every section
can be ``twisted'' such that it will 
extend across the origin of the $j$th copy of $\C^n$.
We use this in the proof of Proposition~\ref{p:lsj}.

\begin{prop} \labell{prop:1} 
Assume the mild technical condition~\eqref{mild technical}.
Given any section $s\in \Gamma_\hol(X_\bfone,\frakL|_{X_\bfone})$,
for every $j=1,\dots,n$
there exists an integer $\ell_j$ such that $t_j^{-\ell_j}\tilde{s}$
extends to a holomorphic section over
$\pi^{-1}(\C^{\times}\times\cdots\times\C^{\times}\times\
\C_{\text{$j$th}}\times\C^{\times}\times\cdots\times\C^{\times})$.
\end{prop}

\begin{proof}

We present the proof of the proposition for the case of $j=1$;
the other cases are similar. 
Set $t = t_1$ and $\hat{t} = (t_2,\ldots,t_n)$.
Thus we need to show that there exists $\ell\in \Z$
such that $t^{-\ell}\tilde{s}$ extends to a holomorphic section
on the subset $\pi^{-1}(\C\times(\C^{\times})^{n-1})$ of $\frakX$. 

For each $\ell \in \Z$,
denote by $U_{\ell}(s)\subset\pi^{-1}(\{0\}\times (\C^{\times})^{n-1})$
the set of those $x \in \pi^{-1}(\{0\}\times (\C^{\times})^{n-1})$
for which $t^{-\ell}\tilde{s}$ extends to some neighbourhood
of $x$ in $\frakX$.
Then $U_{\ell}(s)$ is open in $\pi^{-1}(\{0\}\times(\C^{\times})^{n-1})$.
We will show that $U_{\ell}(s)$ is also closed.
Because $\pi^{-1}(\{0\}\times (\C^{\times})^{n-1})$ is connected
(by Remark~\ref{rk:connected}(a)),
this implies that $U_{\ell}(s)$ is either 
all of $\pi^{-1}(\{0\}\times(\C^{\times})^{n-1})$ or is empty.
We will finish by showing that there exists
$\ell\in \Z$ such that $U_{\ell}(s)\ne \emptyset$. 

For the closedness of $U_{\ell}(s)$, we fix a point $x_0$ 
that belongs to the closure of $U_{\ell}(s)$,
and we will show that $x_0\in U_{\ell}(s)$.

We can write $\pi(x_0) = (0,\hat{t}_0)$ for $\hat{t}_0 \in (\Cx)^{n-1}$.
Fix a neighbourhood $\calN(x_0)$ of $x_0$ in $\frakX$,
and a trivialization of $\frakL$ over $\calN(x_0)$,
and a complex analytic chart with domain $\calN(x_0)$
that has the form
\[ (t,\hat{t}, \hat{z}) \colon \calN(x_0) \rightarrow
   D(0) \times D(\hat{t}_0)\times B,\]
where $D(0)$ is a disk in $\C$ centred at $0$, 
where $D(\hat{t}_0)$ is a disk in $(\Cx)^{n-1}$ centred at $\hat{t}_0$,
where $B$ is a ball in $\C^k$ with $k$ the dimension of the fibre of $\pi$,
and in which the projection $\pi$ 
is the projection to the coordinates $(t,\hat{t})$.

The Laurent expansion of $\tilde{s}$ with respect to the variable $t$
then takes the form
\[\tilde{s}(t, \hat{t}, \hat{z})=\sum_{j\in\Z} C_j(\hat{t}, \hat{z})t^j,\] 
where the coefficients $C_j$
are holomorphic functions in $(\hat{t},\hat{z}) \in D(\hat{t}_0) \times B$.

Since $x_0$ is in the closure of $U_{\ell}(s)$
and $U_{\ell}(s)$ is open, there is 
a non-empty open subset $\calN'$ of $D(\hat{t}_0)\times B$ 
such that $(0, \hat{t}, \hat{z})\in U_{\ell}(s)$ 
for every $(\hat{t}, \hat{z})\in \calN'$. 
For each $j < \ell$, we then have $C_j(\hat{t}, \hat{z})=0$  
for all $(\hat{t}, \hat{z}) \in \calN'$.
By analytic continuation, for each $j < \ell$, 
we have $C_j(\hat{t}, \hat{z})=0$ 
for all $(\hat{t}, \hat{z})\in D(\hat{t}_0)\times B$.
Therefore $t^{-\ell}\tilde{s}$ extends to $\calN(x_0)$,
and thus $x_0\in U_{\ell}(s)$ as expected. 

We have now shown that the set $U_{\ell}(s)$ is both open and closed 
in $\pi^{-1}(\{0\}\times (\C^{\times})^{n-1})$,
so it is either all of $\pi^{-1}(\{0\}\times (\C^{\times})^{n-1})$
or is empty.
To complete the proof of Proposition \ref{prop:1},
we will now show that there exists an $\ell$ 
such that $U_{\ell}(s)\ne\emptyset$. 

Let $x_0\in\pi^{-1}(\{0\}\times (\C^{\times})^{n-1})$ be 
a $(\Cx)_{1^{\text{st}}}$-fixed point 
whose isotropy weights in its fibre are all non-positive;
see Remark~\ref{rk:connected}(b).

There exists a chart 
$(t,\hat{t},\hat{z}) \colon \calN(x_0) 
 \to D(0) \times D(\hat{t}_0) \times B$ as above,
and a trivialization $v \colon \frakL|_{\calN(x_0)} \to \C$
of the line bundle over $\calN(x_0)$,
in which the $(\Cx)_{\text{$1$st}}$ action is
\[ a\cdot (t, \hat{t}, \hat{z}, v) 
   = (at, \hat{t}, a^{m_1}z_1,\dots, a^{m_k}z_k, a^m v),\] 
where $m_1,\ldots,m_k$ are the weights for the 
isotropy $(\Cx)_{\text{$1$st}}$ action on $\pi^{-1}(0,\hat{t}_0)$ 
at $x_0$, and where $m \in \Z$ is the weight
for the $(\Cx)_{\text{$1$st}}$ action on the fibre $\frakL|_{x_0}$. 
(See e.g.\ \cite[Lemma~5 and Corollary 6]{ishida-karshon}.)
By our choice of $x_0$, the weights $m_1,\ldots,m_k$ are non-positive.

The holomorphic section $\tilde{s}$ of $\frakL$ over $\frakX_\reg$,
restricted to the intersection of this neighbourhood
with $\frakX_\reg$, becomes a holomorphic function
on $(D(0) \ssminus \{ 0 \}) \times D(\hat{t}_0) \times B$,
which we write as
\[ v(t,\hat{t}, \hat{z})
 = \sum_{j\in\Z,\ j_1,\dots, j_k\in\Z_{\ge 0}} 
 C_{j,j_1,\dots, j_k} (\hat{t}) t^jz_1^{j_1}\cdots z_k^{j_k}, \] 
where the coefficients $C_{j, j_1,\dots, j_k}$
are holomorphic functions of $\hat{t}$.

Since $\tilde{s}$ is $(S^1)^n$-equivariant, 
\[v(at, \hat{t}, a^{m_1}z_1,\dots, a^{m_k}z_k)
 = a^m v(t, \hat{t}, z_1,\dots, z_k), \] 
which implies that
\[ 
\sum C_{j, j_1,\dots, j_k} (\hat{t}) a^{j+m_1j_1+\cdots+m_kj_k}
 t^jz_1^{j_1} \cdots z_k^{j_k}
 = a^m\sum C_{j, j_1,\dots, j_k}
   (\hat{t})t^jz_1^{j_1}\cdots z_k^{j_k}.  \]

This identity holds on an open set, so the coefficients 
on the left hand side must be equal to those on the right hand side.
Therefore, the only $j, j_1,\dots,j_k$ with $C_{j, j_1,\dots, j_k}\ne 0$ 
are those that satisfy \[j+m_1j_1+\cdots+m_kj_k=m.\]

Thus 
\[v(t, \hat{t}, \hat{z})=\sum_j C_j (\hat{t}, \hat{z}) t^j,\] 
summing over those $j\in \Z$ for which there exist 
$j_1,\dots, j_k\in\Z_{\ge 0}$ with \[j= m-\sum_{i=1}^k m_ij_i.\]
Since all the weights $m_i$ are nonpositive, $j\ge m$.
Therefore, 
$x_0\in U_m(s)$,  
which implies that $U_m(s)\ne\emptyset$, as required.
\end{proof}

\subsection*{The $\ell$-vector} \ 

For each holomorphic section $s \in \Gamma_\hol(X_\bfone,\frakL|_{X_\bfone})$ 
we constructed an equivariant extension $\tilde{s}$,
which is a section of the line bundle $\frakL \to \frakX$ 
over the open dense subset $\frakX_\reg$,
and we showed how to twist this extension $\tilde{s}$ 
so that, after the twist, it further extends to
a global holomorphic section of the line bundle $\frakL$
over the entire family $\frakX$.
In the following proposition
we show that the set of such twists is bounded.
We use this proposition in the proofs 
of Lemma~\ref{l:not too big} and~Proposition~\ref{p:onto}.

\begin{prop} \labell{p:lsj}
Assume the mild technical condition~\eqref{mild technical}.
Let $s \in \Gamma_\hol(X_\bfone,\frakL|_{X_{\bfone}})$
be a non-zero holomorphic section.
Then for each $j \in \{ 1,\ldots,n\}$
there exists a unique integer $\ell_{s,j}$ such that the set
\begin{multline} \labell{Nsj}
\{ \ell_j \in \Z \ | \ 
t_j^{-\ell_j} \tilde{s}
\text{ extends to } \\
\text{ a holomorphic section of $\frakL$ over }
\pi^{-1} ( \Cx \times \cdots \times \Cx \times \C_{\text{$j$th}} 
\times \Cx \times \cdots \times \Cx ) \}\
\end{multline}
coincides with the set
$$ \{ \ell_j \ | \ \ell_j \leq \ell_{s,j} \}.$$
\end{prop}

\begin{Definition} \labell{l-vector}
In the setup of Proposition~\ref{p:lsj},
the \textbf{$\ell$-vector} of~$s$ is 
$$ \vell_s := ( \ell_{s,1} , \ldots , \ell_{s,n} ) .$$
\end{Definition}

\begin{proof}[Proof of Proposition~\ref{p:lsj}]
From Proposition \ref{prop:1}, we know that the set \eqref{Nsj} is non-empty. 
Since $s$ is non-zero, the set \eqref{Nsj} is a proper subset of $\Z$.
(If $\tilde{s}$ does not extend to the open set
$\pi^{-1} ( \Cx \times \cdots \times \Cx \times \C_{\text{$j$th}}
\times \Cx \times \cdots \times \Cx )$,
then the exponent $0$ is not in the set~\eqref{Nsj}.
Otherwise, examine the extension $\ol{\tilde{s}}$ 
near a point where the $j$th coordinate vanishes.
At that point, if $\ol{\tilde{s}}$ has a zero,
then the zero must be of finite order.
So, if we multiple the section $\tilde{s}$ 
by a sufficiently negative power of $t_j$,
then we'll get a section that doesn't extend.)
Furthermore, if $\ell_j$ is in the set \eqref{Nsj} and $\ell_j'\le \ell_j$,
then $\ell_j'$ is in the set \eqref{Nsj}.
These facts imply the conclusion of the proposition.
\end{proof}

Recall that
to any lattice vector $\vec{\ell} = (\ell_1,\ldots,\ell_n)$ in $\Z^n$
the filtration associates the space $F_{\vell}$
that consists of those sections $s$ such that 
$\bft^{-\vell}\tilde{s}$ extends to a global holomorphic section
of $\frakL \to \frakX$.
For a holomorphic section $s$ with an $\ell$-vector $\ell_s$
(Definition~\ref{l-vector}),
Propositions~\ref{prop:all j} and~\ref{p:lsj} imply the following criterion, 
which we use in the proof of Proposition~\ref{p:onto}:
\begin{equation} \labell{eq:Fl}
s \in F_{\vell} \quad \text{ if and only if } \quad \vell \leq \vell_s \,,
\end{equation}
with respect to the product partial ordering~\eqref{product partial ordering}.

\subsection*{Leaves embed} \ 

Assuming that the higher cohomology
of the sheaf of holomorphic sections of $\frakL$ vanishes,
we will now show how to embed each ``leaf'' $F_{\vell}/F_{>\vell}$ 
into the space of holomorphic sections over the special fibre $X_{\bfzero}$.
We begin with a technical lemma, which we use in the proof 
of Proposition~\ref{map to centre}.

\begin{Lemma} \labell{l:hadamard}
Assume that the higher cohomology 
of the sheaf of holomorphic sections of $\frakL \to \frakX$ vanishes.
Let $s$ be a holomorphic section 
in $\Gamma_\hol(X_\bfone,\frakL|_{X_\bfone})$.
Let $\vell \in \Z^n$.
Suppose that the section $\bft^{-\vell} \tilde{s}$ 
extends to a section in $\Gamma_\hol(\frakX,\frakL)$
whose restriction to the special fibre $X_\bfzero = \pi^{-1}(\bfzero)$ 
vanishes.
Then there exist holomorphic sections 
$s_1,\ldots,s_n \in \Gamma_\hol(X_\bfone,\frakL|_{X_\bfone})$
such that $s = s_1 + \ldots + s_n$
and such that, for every $j = 1, \ldots, n$,
the section $t_j^{-1}\bft^{-\vell} \tilde{s}_j$ extends 
to a section in $\Gamma_\hol(\frakX,\frakL)$.
\end{Lemma}

\begin{proof}
Since we are assuming that the higher cohomology 
of the sheaf of holomorphic sections of $\frakL$ vanishes,
Theorem~\ref{HA} of Appendix~\ref{app:hadamard}
allows us to write
$\bft^{-\vell} \tilde{s} = \sum_{j=1}^n t_j \sigma_j'$,
where $\sigma_j'$ are holomorphic sections in $\Gamma_\hol(\frakX,\frakL)$.
By applying an averaging argument to $s_j' := t_j \bft^{\vell} \sigma_j'$
(see below), 
we obtain sections $\tilde{s}_j$ that are $(\Cx)^n$ equivariant,
such that $\tilde{s} = \tilde{s}_1 + \ldots + \tilde{s}_n$
and such that $t_j^{-1} \bft^{-\vell} \tilde{s}_j$ extend to $\frakX$.
We then set $s_j := \tilde{s}_j|_{X_\bfone}$,
so that $\tilde{s}_j$ is the sweep of $s_j$.

(Here is the averaging argument.  Let
$$ \tilde{s}_j(x) := \int_{a \in (S^1)^n} s_j'(a \cdot x) da.$$
Then $\tilde{s}_j$ is holomorphic and $(S^1)^n$-invariant,
so it is $(\Cx)^n$-invariant. Writing
\begin{multline*}
(t_j^{-1} \bft^{-\vell} \tilde{s}_j)(x)
 =  t_j(x)^{-1} \bft^{-\vell}(x) \int_{a \in (S^1)^n}
    t_j(a \cdot x) \bft^{\vell}(a \cdot x) \sigma_j'(a \cdot x) da
 = \int_{a \in (S^1)^n}  a_j a^{\vell} \sigma_j'(a \cdot x)  da \, , 
\end{multline*}
we obtain that $t_j^{-1}\bft^{-\vell} \tilde{s}_j$ extends
to a holomorphic section over $\frakX$.)
\end{proof}

The following proposition is a rephrasing of Part (a) of Theorem \ref{t:main}.
We also use it in the proof of Proposition~\ref{p:one-to-one}.

\begin{Proposition} \labell{map to centre}
For every $\vec{\ell} \in \Z^n$, we have the $H$-equivariant linear map 
given by ``sweep, twist, extend, restrict'':
$$ \Phi \colon F_{\vell} 
   \to \Gamma_\hol (X_\bfzero , \frakL|_{X_{\bfzero}} ) \, , \quad 
 s \mapsto 
 \left.\left(\text{extension to } \frakX
 \text{ of } \bft^{-\vell} \tilde{s} \right)\right|_{X_0}.$$
\bigskip
The kernel of this map contains $F_{> \vell}$,
and it is equal to $F_{>\vell}$
if the higher cohomology of the sheaf of holomorphic sections 
of $\frakL$ vanishes.
\end{Proposition}

\begin{proof} Let $s\in F_{\vec{\ell}} $. 
Then $\mathbf{t}^{-\vec{\ell}}\tilde{s}$ extends to a unique 
global holomorphic section of the line bundle $\frakL$ over $\frakX$,
namely, $\ol{\mathbf{t}^{-\vec{\ell}}\tilde{s}}$.  
We denote by $s_{0}$ the restriction of this extended holomorphic section 
to the special fibre $X_{\mathbf{0}}$.  This defines a linear map  
\[ \Phi \colon F_{\vec{\ell}} 
    \rightarrow \Gamma_\hol(X_{\mathbf{0}}, \frakL|_{X_{\mathbf{0}}})
 \, , \quad s\mapsto s_0.  \]

Let $s \in F_{>\vell}$.
Let $\vell' > \vell$ be such that $s \in F_{\vell'}$.
Then 
 \[ s_0(x) = \ol{\mathbf{t}^{-\vec{\ell}} \tilde{s}} (x)
 = \mathbf{t}^{\vec{\ell'}-\vec{\ell}} 
    \cdot \ol{\mathbf{t}^{-\vec{\ell'}} \tilde{s}} (x)
 = \mathbf{t}^{\vec{\ell'}-\vec{\ell}}(x) 
    \cdot \ol{\mathbf{t}^{-\vec{\ell'}}\tilde{s}} (x)
 = 0 \cdot \ol{\mathbf{t}^{-\vec{\ell'}}\tilde{s}} (x)
 = 0,\] 
since $\pi(x)=(0,\dots, 0)$
and the exponents $\ell'_j-\ell_j$ are all non-negative
and at least one of them is positive.
This shows that $F_{>\vec{\ell}} \subseteq \ker(\Phi)$. 

Finally, suppose that the higher cohomology of the sheaf
of holomorphic sections of $\frakL$ vanishes,
and let $s \in \ker(\Phi)$.
Then the extension
of $\bft^{-\vell} \tilde{s}$ vanishes along $X_0$.
Lemma~\ref{l:hadamard}
implies that we can write $s = s_1 + \ldots + s_n$
where each $s_j$ belongs 
to $F_{(\ell_1,\ldots,\ell_{j-1},\ell_j+1,\ell_{j+1},\ldots,\ell_n)}$.
Because $(\ell_1,\ldots,\ell_{j-1},\ell_j+1,\ell_{j+1},\ldots,\ell_n)
 > \vell$, we get that $s \in F_{>\vell}$.
\end{proof}

\subsection*{Decomposition of the space of sections} \ 

Propositions~\ref{p:onto} and~\ref{p:one-to-one} below
constitute the proof of Part~(b) of Theorem~\ref{t:main}.
The proof of Proposition~\ref{p:onto} uses the following technical lemma.

\begin{lemma} \labell{l:not too big}
Assume the mild technical condition~\eqref{mild technical}.
Then, for each $j \in \{ 1 , \ldots , n \}$, the set
$$ \left\{ \, \ell_{s,j} \ | \ 
   0 \neq s \in \Gamma_\hol(X_\bfone,\frakL|_{X_\bfone}) \, \right\} $$
is bounded from above. \\
Here, $\ell_{s,j}$ are the integers from Proposition~\ref{p:lsj}.
\end{lemma}

\begin{proof}
The condition ``$t_j^{-\ell_j} \tilde{s}$ does not extend 
to a holomorphic section of $\frakL$ over
$\pi^{-1} ( \Cx \times \cdots \times \Cx \times \C_{\text{$j$th}}
\times \Cx \times \cdots \times \Cx ) $''
is an open condition in $s$.

(Here are some details. For a holomorphic section in $\Gamma_\hol(\mathfrak{X}_\reg,\frakL|_{X_\reg})$,
the property 
of not extending to a holomorphic section of $\frakL$ over
$\pi^{-1} ( \Cx \times \cdots \times \Cx \times \C_{\text{$j$th}}
\times \Cx \times \cdots \times \Cx ) $
means that  in local coordinates near some point
in $\{ t_j = 0 \}$, 
certain coefficients in certain Laurent expansions are non-zero.
By choosing a basis of  $\Gamma_\hol(X_1,\frakL|_{X_1})$ (which is finite dimensional since $X_1$ is compact), we can identify $\Gamma_\hol(X_1,\frakL|_{X_1})$ with $\C^N$ for some $N\in\Z_{\ge 0}$. 
From Cauchy's integral formula
we see that these coefficients are continuous 
with respect to the  topology
on $\C^N$.  
)

It follows that the map $s \mapsto \ell_{s,j}$
from $\Gamma_\hol( X_\bfone , \frakL|_{X_\bfone} ) \ssminus \{ 0 \} $ to $\Z$
is upper semi-continuous.
Since this map descends to a map 
$ \PP(\Gamma_\hol(X_\bfone,\frakL|_{X_\bfone})) \to \Z$ 
whose domain is compact,
we conclude that this map is bounded from above.
\end{proof}

\begin{Proposition} \labell{p:onto}
Assume that the mild technical condition~\eqref{mild technical} holds.
Then the map $V_\graded \to V$ 
that is defined in Theorem~\ref{t:main}(b) is onto.
\end{Proposition}

\begin{proof}
For each $j \in \{ 1, \ldots , n \}$,
let $b_j$ be an integer upper bound to the set 
$\{ \ell_{s,j} \ | \ 
    0 \neq s \in \Gamma_\hol( X_\bfone, \frakL|_{X_\bfone} ) ) \}$,
where $\ell_{s,j}$ are the integers from Proposition~\ref{p:lsj};
such a bound exists by Lemma~\ref{l:not too big}.

We need to prove that every non-zero vector $s \in V$
is in the image of the map $V_\graded \to V$.
We argue by induction on the non-negative integer
$$\delta_s := (b_1+\ldots+b_n) - (\ell_{s,1}+\ldots+\ell_{s,n}).$$

Let $\delta$ be a non-negative integer.
Assume that, for every non-zero vector 
$s' \in \Gamma_\hol(X_\bfone,\frakL|_{X_\bfone})$,
if $\delta_{s'} < \delta$,
then $s'$ is in the image of $V_\graded \to V$. (If $\delta=0$, this assumption is vacuously true.) 
Let $s \in V$ be a non-zero vector with $\delta_s = \delta$.
We would like to show that $s$ is in the image of $V_\graded \to V$.

Let $\vell_s = (\ell_{s,1},\ldots,\ell_{s,n})$ be the $\ell$-vector of $s$.  
Then $s$ is in $F_{\vell_s}$. 
(See Definition~\ref{l-vector} and the criterion~\eqref{eq:Fl}.)

Because $V_{\vell_s}$ is a complementary subspace to $F_{>\vell_s}$
in $F_{\vell}$,
we can write $s = s_0 + s''$, 
where $s_0 \in V_{\ell_s}$ and $s'' \in F_{> \ell_s}$.

By the definition of $F_{> \ell_s}$,
we can write $s'' = s_1 + \ldots + s_n$ 
where $\ell_{s_i,j} \geq \ell_{s,j}$ for all $i,j$
and $\ell_{s_j,j} > \ell_{s,j}$ for all $j$.
So for each $j$ we have $\delta_{s_j} < \delta_s$,
and by the induction hypothesis $s_j$ is in the image 
of the map $V_\graded \to V$. 
But $s_0$, being in $V_{\ell_s}$, 
is also in the image of $V_\graded \to V$.
So $s = s_0 + s_1 + \ldots + s_n$, 
being a sum of elements in the image of $V_\graded \to V$,
is also in this image.
\end{proof}

\begin{Proposition} \labell{p:one-to-one}
Assume that the higher cohomology
of the sheaf of holomorphic sections of $\frakL$ vanishes.
Then the map $V_\graded \to V$ 
that is defined in Theorem~\ref{t:main} is one-to-one.
\end{Proposition}

\begin{proof}
First, we claim that, for any $m \in \N$,
if $\vell^{(1)}, \ldots, \vell^{(m)}$ are distinct,
$s_j \in F_{\vell^{(j)}}$ for each $j=1,\ldots,m$,
and $s_1 + \ldots + s_m = 0$,
then there exists at least one $j$ such that $s_j \in F_{> \vell^{(j)} }$.

Indeed,
let $\vell^{(1)},\ldots,\vell^{(m)}$ be distinct elements of $\Z^n$,
let $s_1,\ldots,s_m \in V = \Gamma_\hol(X_\bfone,L_\bfone)$,
and suppose that $s_j \in F_{\vell^{(j)}}$ for each $j$.
This means that, for each $j$, there exists 
$\sigma_j \in \Gamma_{\hol}(\frakX,\frakL)_{\vell^{(j)}}$
such that $\tilde{s}_j = \bft^{\vell^{(j)}} \sigma_j|_{\frakX_\reg}$.
Suppose that $s_1 + \ldots + s_m = 0$.
``Sweeping'', $\tilde{s}_1 + \ldots + \tilde{s}_m = 0$.
By Lemma~\ref{one to one} of Appendix~\ref{app:indep},
there exists $j$ such that $\sigma_j$ vanishes on $X_\bfzero$.
Since we are
assuming that $H^{\geq 1}(\frakX,\calO_{\frakL}) = \{ 0 \}$,
where $\calO_{\frakL}$ the sheaf of holomorphic sections of $\frakL$,
Proposition~\ref{map to centre}
implies that $s_j \in F_{> \vell^{(j)}}$.
This completes the proof of the claim.

A vector in $V_\graded$ can be written 
as $\left\{ s_{\vell} + F_{>\vell} \right\}_{\vell \in \Z^n}$,
where $s_{\vell} \in V_{\vell}$ for each $\vell$
and all but finitely many of the components $s_{\vell}$ are zero;
its image in $V$ is the sum $\sum s_{\vell}$.
We will prove, by induction on $m$, that for every such vector
in which there are at most $m$ non-vanishing components, 
if $\sum s_{\vell} = 0$ then $s_{\vell} = 0$ for all $\vell$.
For $m=0$, this is vacuously true.  Assume that it is true for $m-1$.
Let $\vell^{(1)}, \ldots, \vell^{(m)}$ be distinct elements of $\Z^n$,
let $s_i := s_{\vell^{(i)}} \in V_{\vell^{(i)}}$ for each $i=1,\ldots,m$,
and assume that $s_1 + \ldots + s_m = 0$.
By the claim, there exists at least one $j$ 
such that $s_j \in F_{> \vell^{(j)} }$.
But $s_j$ is in $V_{\vell^{(j)}}$,
which is complementary to $F_{> \vell^{(j)} }$ in $F_{\vell^{(j)}}$, 
so $s_j = 0$.
By the induction hypothesis applied to $\{ s_i \}_{i=1}^m \ssminus \{ s_j \}$,
we obtain that $s_i = 0$ for all $i$, as required.
Thus, the kernel of the map $V_\graded \to V$ is trivial,
and so the map is one-to-one.
\end{proof}

\begin{Remark}[Valuations] \labell{rk:valuations}
The filtration of the vector space 
$V = \Gamma_\hol(X_\bfone, \frakL|_{X_\bfone})$ 
into the spaces $F_{\vell}$ 
satisfies the properties of a prevaluation 
in the sense of Kaveh-Khovanskii \cite[Definition~2.1]{kaveh-khovanskii},
except that our indexing set is only partially ordered
and not totally ordered.
Namely, the function $s \mapsto \vell_s$,
where $\vell_s$ is as in Definition~\ref{l-vector},
has the following two properties.
For all $s_1,s_2$ such that $s_1,s_2,s_1+s_2$ are non-zero,
$\vell_{s_1+s_2} \geq \min ( \vell_{s_1} , \vell_{s_2} )$.
And, for every $s \in V \ssminus \{ 0 \}$ 
and $\lambda \in \C \ssminus \{0\}$, we have $\vell_{\lambda s} = \vell_s$. The study of valuations is  important in the theory of Newton-Okounkov bodies\cite{kaveh-khovanskii}. Several examples of valuations and Newton-Okounkov bodies for Bott Samelson manifolds occur in literature, including Kaveh \cite{kaveh} and  Harada-Yang \cite{harada-yang}. It would be interesting to  find connections between the notion of $\ell$-vector in our work (Definition \ref{l-vector}) and these known examples.

\end{Remark}

\section{Bott Samelson Equivariant Families}
\labell{sec:Bott-Samelson}

In this section, we construct equivariant families
in which the Bott-Samelson manifold is the generic fibre
and the corresponding Bott tower is the special fibre. We will use this construction in Section ~\ref{sec:canonical-basis} to make a connection to representation theory.

\subsection*{Line bundles over Bott-Samelson manifolds}\ 
\label{sec:bott}

We recall the setup from Section~\ref{sec:overview}.
Let $K$ be a compact connected Lie group.
Fix a maximal torus $T$ in $K$.  Let $G$ be the complexification of $K$,
and let $H$ be the Cartan subgroup, which is the complexification of $T$.
Consider the Lie algebra $\t =\Lie(T)$ of $T$
and the weight lattice $\t^*_{\Z}$ in the dual $\t^*$ of $\t$. 
Let $\Delta\subset\t^*_{\Z}$ be the set of roots of $K$ with respect to $T$.
Choose a Borel subgroup $B$ of $G$ that contains $T$.
Consider the resulting set $\Delta_+$ of positive roots in $\Delta$,
positive Weyl chamber $\t^*_+$ in $\t^*$, 
Bruhat (partial) order on $\t^*_{\Z}$,
and set $\{\alpha_1,\ldots,\alpha_r\}$
of simple positive roots in $\Delta_+$.
Consider the projection $\psi_H \colon B \to H$
whose kernel is the maximal unipotent subgroup $U$ of $B$.

For each simple positive root $\alpha_i$, let $P_{\alpha_i}$ be 
the corresponding minimal parabolic subgroup of~$G$.
Fix a sequence $\alpha_{i_1},\ldots,\alpha_{i_n}$ of simple positive roots.
In Section~\ref{sec:overview} we defined the Bott-Samelson manifold
$Z_{\alpha_{i_1},\ldots,\alpha_{i_n}}$
and the Bott tower $X_{\alpha_{i_1},\ldots,\alpha_{i_n}}$,
both obtained as quotients of 
$P_{\alpha_{i_1}} \times \cdots \times P_{\alpha_{i_n}}$
by $B^n$ actions where each factor of $B$ acts on two consecutive
parabolics (see \eqref{Zalpha} and \eqref{Xalpha}).
For any $n$-tuple $(\mu_1,\ldots,\mu_n)$
of elements of the weight lattice $\t^*_\Z$,
let $\C_{\mu_1,\dots, \mu_n}$ 
be the one-dimensional representation of $B^n$ 
that is obtained as the projection $B^n \to H^n$
followed by the $H^n$ action with weight $(\mu_1,\ldots,\mu_n)$.
Then we can consider the line bundles
\begin{equation} \labell{line bundles}
 L_{\mu_1,\ldots,\mu_n} \to Z_{\alpha_{i_1},\ldots,\alpha_{i_n}}
\text{\qquad and \qquad }
 L_{\mu_1,\ldots,\mu_n}^X \to X_{\alpha_{i_1},\ldots,\alpha_{i_n}} \ ,
\end{equation}
each defined by $ (P_{\alpha_{i_1}} \times \ldots \times P_{\alpha_{i_1}})
   \times_{B^n}  \C_{-\mu_1,\ldots,-\mu_n} $
with the appropriate
right $B^n$ action on the product of the parabolics.
These holomorphic line bundles are $B$ equivariant
with respect to the left $B$ action 
that is obtained from left multiplication on the $P_{\alpha_{i_1}}$ factor.
In the special case that $\mu_1 = \ldots \mu_{n-1} = 0$
and $\mu_n = \lambda$,
we obtain the line bundles $L_\lambda^Z$ and $L_\lambda^X$
from Section~\ref{sec:overview}
(see~\eqref{L lambda Z} and~\eqref{L lambda X}).

\begin{Remark} \labell{rk:LLM}
Lakshmibai-Littelmann-Magyar \cite{lakshmibai-littelmann-magyar}
considered, for any ${\bf{m}} = (m_1,\dots, m_n) \in \Z_{\ge 0}$,
the line bundle
\[ L_{\bf{m}} := L_{\mu_1,\ldots,\mu_n},\]
where $\mu_1 = m_1 \varpi_{i_1}$,\ $\dots$,\ $\mu_n = m_n \varpi_{i_n}$
are the corresponding integer multiples 
of the fundamental weights $\varpi_{i_1},\dots, \varpi_{i_n}$. 
The space of holomorphic sections of this line bundle
is isomorphic, as a $B$ module, to a so-called generalized Demazure module;
see \cite[Theorem~6]{lakshmibai-littelmann-magyar}.
A special case of a Demazure module is the restriction to $B$ 
of the irreducible representation $V_\lambda$ of~$G$ 
with maximal weight $\lambda$.
This representation can be modeled by the space of holomorphic sections 
$\Gamma_\hol(Z_{\alpha_{i_1},\ldots,\alpha_{i_n}}, L_{0,\ldots,0,\lambda})$
if the sequence of simple positive roots 
$\alpha_{i_1}, \ldots, \alpha_{i_n}$ 
corresponds to a reduced expression 
for the longest element of the Weyl group.
The representation $V_\lambda$ of $B$ can also be modeled
by the space of holomorphic sections of the line bundle $L_{\bf{m}}$,
where $m_k = \left< \lambda, \alpha_{i}^{\vee} \right>$
if $k$ is the largest index such that $i_k=i$,
and $m_k=0$ if there isn't such an $i$.
See Lakshmibai-Littelmann-Magyar \cite[page~294]{lakshmibai-littelmann-magyar}
and Lauritzen-Thomsen \cite[Section~3]{lauritzen-thomsen}.
\end{Remark}

\subsection*{Bott-Samelson equivariant families}\ 

Michael Grossberg \cite{grossberg:thesis,grossberg-karshon}
described a deformation of the $B^n$ action 
on $P_{\alpha_{i_1}} \times \ldots \times P_{\alpha_{i_n}}$
that gives an interpolation between the Bott-Samelson manifold
and the corresponding Bott tower.
We now fit this deformation into an equivariant family over $\C^n$
whose fibre over $\bfone = (1,\ldots,1)$ is the complex Bott-Samelson manifold
and whose special fibre is the corresponding Bott tower:

\begin{Theorem}
\labell{extension of toric action}
There is an $H$-equivariant family \[\frakL_{\mu_1,\ldots,\mu_n} \to \frakX_{\alpha_{i_1},\ldots,\alpha_{i_n}}  \to \C^n\]
(in the sense of \S\ref{sec:filtrations})
that satisfies the following properties.  (See the notations below.)

\begin{enumerate}
\item\labell{intertwining}
There exist isomorphisms of holomorphic line bundles
$$
\xymatrix{
 L_{\bfzero} \ar[d] \ar[rr] && L_{\mu_1,\ldots,\mu_n}^X \ar[d] \\
 X_{\bfzero} \ar[rr] && X_{\alpha_{i_1},\ldots,\alpha_{i_n}} \\
}
\qquad \text{ and } \qquad
\xymatrix{
 L_{\bfone} \ar[d] \ar[rr] && L_{\mu_1,\ldots,\mu_n} \ar[d] \\
 X_{\bfone} \ar[rr] && Z_{\alpha_{i_1},\ldots,\alpha_{i_n}} \\
}
$$
that intertwine the $H$ actions on $L_\bfzero$ and on $L_\bfone$ 
with the $H$ actions on $L_{\mu_1,\ldots,\mu_n}^X$
and on $L_{\mu_1,\ldots,\mu_n}$,
and the first of these also intertwines
the $(\Cx)_{\text{$j$th}}$ action on $X_\bfzero$ with the action
\begin{equation} \labell{jth}
\tau \cdot [p_1,\ldots,p_n] = 
              [p_1,\ldots,p_{j-1},S(\tau)p_j,p_{j+1},\ldots,p_n]
\end{equation}
on $X_{\alpha_{i_1},\ldots,\alpha_{i_n}}$.
 
\item\labell{mild tech} The family satisfies the mild technical condition  (\ref{mild technical}).

\item \labell{unique}
If $G$ is semisimple, (in particular, if $K$ is simply connected,) 
the construction of the family can be made to depend only on the 
choices of a maximal torus $T$ of $K$,
a Borel subgroup $B$ of $G$ that contains $T$,
and a reduced expression for the longest element of the Weyl group.
If $G$ is reductive, these choices determine the family 
up to tensoring the fibres of the line bundle 
by a one-dimensional representation of $(\Cx)^n$.
\end{enumerate}
\end{Theorem}

We prove Theorem \ref{extension of toric action}
at the end of this section.

In Theorem \ref{extension of toric action},  we used the following notations. The special fibre $X_{\bfzero}$ is the preimage in  $\frakX_{\alpha_{i_1},\ldots,\alpha_{i_n}}$
of the origin $\bfzero := (0,\ldots,0)$ of $\C^n$,
and the fibre $X_{\bfone}$ is the preimage in $\frakX_{\alpha_{i_1},\ldots,\alpha_{i_n}}$ 
of the point $\bfone := (1,\ldots,1)$ of $\C^n$.
The line bundle $L_{\bfzero}$ over $X_{\bfzero}$
 is the preimage of $\bfzero$ in $\frakL_{\mu_1,\ldots,\mu_n}$,
and the line bundle $L_{\bfone}$ over $X_{\bfone}$ is the preimage of $\bfone$ in $\frakL_{\mu_1,\ldots,\mu_n}$.
Note that the $H \times (\Cx)^n$ action on the equivariant family
restricts to an action on the line bundle $L_\bfzero \to X_\bfzero$
and that the $H$ action on the equivariant family
restricts to an action on the line bundle $L_\bfone \to X_\bfone$.
The line bundles $L_{\mu_1,\ldots,\mu_n}^X$ 
and $L_{\mu_1,\ldots,\mu_n}$
are as in \eqref{line bundles}.

\begin{Lemma} (\cite[\S 1]{pasquier})
\labell{one-parameter subgroup existence}
There exist a one-parameter subgroup of the Cartan subgroup,
$$ S \colon \Cx \to H ,$$
and a positive integer $q$,
such that, for any simple positive root $\alpha_i$,
the composition of $ S \colon \C^{\times} \rightarrow H $
with the homomorphism $H \to \Cx$ that is represented by $\alpha_i$ 
is $t \mapsto t^{q}$.
\end{Lemma}

\begin{proof}
Let $\t^*_\Q$ be the rational span of the weight lattice $\t^*_\Z$,
and let $\t_\Q$ be the rational span of the dual lattice $\t_\Z$.
Note that $\t_\Q = \bigcup\limits_{q \in \N} \frac{1}{q} \t_\Z$.
Because the simple positive roots $\alpha_1,\ldots,\alpha_r$
are linearly independent and are (in $\t^*_\Z$, hence) in $\t^*_\Q$,
for every $r$-tuple $q_1',\ldots,q_r'$ of rational numbers
there exists an element $\eta$ of $\t_\Q$ whose pairing with each $\alpha_i$ is $q_i'$.
Take ${q_1'=\ldots=q_r'=1}$, let $q$ be a positive integer such that
the corresponding $\eta$ is in $\frac{1}{q} \t_\Z$,
and take $S \colon \Cx \to H$ to be the one-parameter subgroup
that is determined by the element $q \eta$ of $\t_\Z$.
Then $S$ and $q$ are as required.
\end{proof}

We use the first part of the following lemma
to construct the Bott-Samelson family;
we use the second part to explain, in Remark~\ref{different S},
the extent that the Bott-Samelson family depends on the choice of $S$;
we use the third part to prove, in Lemma~\ref{technical assumption holds},
that the Bott-Samelson family satisfies 
the mild technical condition~\eqref{mild technical}.

\begin{Lemma} \labell{one-parameter subgroup properties}
Let $S \colon \Cx \to H$ be a one-parameter subgroup of the Cartan subgroup
such that, for every simple positive root $\alpha_i$,
the composition of $ S \colon \C^{\times} \rightarrow H $
with the homomorphism $H \to \Cx$ that is represented by $\alpha_i$ 
has the form $t \mapsto t^{q_i}$ for some positive integer~$q_i$.
Then the following is true.

\begin{itemize}[itemsep=6pt]
\item
The map $(t,b) \mapsto \psi_t(b)$ from $\C \times B$ to $B$ that is given by 
\begin{equation*} 
 \begin{array}{ll}
   \psi_{t}(b) & := S(t) b S(t)\inv \text{ if } t\ne 0, \text{ and } \\
   \psi_0 & := \psi_H , 
                \text{ the projection $B \to H$ with kernel $U$},
\end{array} 
\end{equation*}
is holomorphic.

\item 
If $\wh{S} \colon \Cx \to H$ is another one-parameter subgroup,
with the same properties as $S$ for the same positive integers $q_i$,
then $\wh{S}$ and $S$ differ by a one-parameter subgroup of the centre of $G$.  
Consequently, 
$S$ and $\wh{S}$ give the same map $\psi_t$.

\item
For every simple positive root $\alpha_i$,
let $P_{\alpha_i}$ be the corresponding minimal parabolic subgroup of~$G$,
and let $\Cx$ act on $P_{\alpha}/B$ by 
$\tau \colon [p] \mapsto [S(\tau)p]$.
Then the isotropy weight for the linearized action
on the tangent space at the identity coset $[e]$ is $-q_i$.

\end{itemize}
\end{Lemma}

\begin{proof}
Let $\b$ be the Lie algebra of the Borel subgroup $B$.
For each $t \in \C$, define $\wt{\psi}_t \colon \b \to \b$ as follows.
Write each element $\eta$ of $\b$
as $\eta = \eta_\h + \sum_{\alpha \in \Delta_+} \eta_\alpha$
with $\eta_\h \in \h$ and $\eta_\alpha \in \g_{\alpha}$,
and define $\wt{\psi}_t(\eta) = \eta_\h + \sum_\alpha t^{q(\alpha)} \eta_\alpha$
where $q(\alpha) = m_1 q_1 + \ldots + m_r q_r$
if $\alpha = m_1 \alpha_1 + \ldots + m_r \alpha_r$.
Because every positive root $\alpha$ is a linear combination of the 
simple positive roots $\alpha_i$
with non-negative coefficients that are not all zero,
the exponents $q(\alpha)$ are all positive,
and the map $(t,\eta) \mapsto \wt{\psi}_t(\eta)$ from $\C \times \b$ to $\b$
is holomorphic.  Holomorphicity of the map $(t,b) \mapsto \psi_t(b)$
then follows from the commuting diagram 
$$
\xymatrix{
 \C \times \b \ar[rr]^{ \quad (t,\eta) \mapsto \wt{\psi}_t(\eta) } 
              \ar[d]^{\text{Id} \times \exp} 
 && \b \ar[d]^{\exp} \\ 
 \C \times B \ar[rr]^{ \quad (t,b) \mapsto \psi_t(b) } && B 
}
$$
and from the exponential map $\exp \colon \b \to B$ 
being a covering map of complex manifolds.

follows from the fact that the centre of a connected reductive complex Lie group
is the intersection of the kernels of the homomorphisms $H \to \Cx$
that correspond to the simple positive roots $\alpha_i$. (See ~\cite[Theorem~5.20]{kamnitzer}.)

The last part of the lemma follows from the $T$-equivariant isomorphisms
\begin{equation} \labell{g -alpha}
T_{[e]} (P_{\alpha_i}/B) \cong \g_{-\alpha_i}.
\end{equation}
These, in turn, follow from our convention 
for the definitions of $B$ and $P_{\alpha_i}$:
the Lie algebra of $B$ is 
$\h \oplus \bigoplus_{\alpha \in \Delta_+} \g_{\alpha}$
and the Lie algebra of $P_{\alpha_i}$ is
$\h \oplus \bigoplus_{\alpha \in \Delta_+} \g_{\alpha} \oplus \g_{-\alpha_i}$,
where $\h$ is the Lie algebra of the Cartan subgroup,
$\Delta_+$ is the set of positive roots,
and, for each root $\alpha$, 
the space $\g_{\alpha}$ is the corresponding root space.
\end{proof}

\begin{Construction} \labell{actions}
Fix a one parameter subgroup 
$$ S \colon \Cx \to H $$
and positive integers $q_1,\ldots,q_r$
as in Lemma~\ref{one-parameter subgroup properties}.
(By Lemma~\ref{one-parameter subgroup existence},
there exist such $S$ and $q_i$ with $q_1=\ldots=q_r$.)
Let $\psi_t \colon B \to B$ be 
as in Lemma~\ref{one-parameter subgroup properties}.
Fix a sequence of simple positive roots, 
$\alpha_{i_1}$, $\ldots$, $\alpha_{i_n}$
and an $n$-tuple $(\mu_1,\ldots,\mu_n)$
of elements of the weight lattice $\t_\Z^*$.
On $\C^{n}\times P_{\alpha_{i_1}}\times\cdots\times P_{\alpha_{i_n}} 
 \times \C_{-\mu_1,\ldots,-\mu_n}$,
we define a right $B^{n}$ action by 
\begin{multline}\labell{B action}
(t_1, \, \dots, \, t_{n}; \, p_1,\, \dots, \, p_n; \, z)
     \cdot (b_1,\, \dots, \, b_n ) \\
   = (t_1, \, \dots, \, t_{n}; \,
     p_1b_1, \, \psi_{t_2}(b_1)\inv p_2 b_2, \, \psi_{t_3}(b_2)\inv p_3 b_3, \,
             \dots, \, \psi_{t_{n}}(b_{n-1})\inv p_n b_n;
     \, (b_1,\ldots,b_n)^{-1} \cdot z)
\end{multline}
(with $(b_1,\ldots,b_n)^{-1} \cdot z = b_1^{\mu_1} \cdots b_n^{\mu_n} z )$),
and we define a left $H \times (\C^{\times})^{n}$ action by 
\begin{multline}\labell{H action}
( h, \, \tau_1, \, \dots, \, \tau_{n} ) \cdot 
    (t_1, \, \dots, \, t_{n}, \, p_1, \, \dots, \, p_n ; \, z)  \\
  = (\tau_1t_1, \, \dots, \, \tau_{n}t_{n}; \,
 h S(\tau_1) p_1, \, S(\tau_2)p_2, \, \dots, \, 
                           S(\tau_{n})p_n ; \, z).
\end{multline}
\end{Construction}

\begin{lemma} \labell{lem:commuting actions}
The right $B^n$ action \eqref{B action} and left $H\times (\C^{\times})^{n}$ action \eqref{H action}
commute with each other.
Moreover, the right $B^n$ action~\eqref{B action},
and the corresponding right $B^n$ action on
$\C^{n}\times P_{\alpha_{i_1}}\times\cdots\times P_{\alpha_{i_n}}$,
are free and proper. 
\end{lemma} 

\begin{proof} We need to show that 
\begin{multline*}
(h, \tau_1,\dots, \tau_{n}) \cdot 
  ((t_1,\dots, t_{n}; p_1,\dots, p_n; z) 
      \cdot (b_1,\dots, b_n)) \\
 = ((h, \tau_1,\dots, \tau_{n}) \cdot (t_1,\dots, t_{n}; p_1,\dots, p_n; z)) 
 \cdot (b_1,\dots, b_n) ,
\end{multline*}
where $(t_1,\dots, t_{n};p_1,\dots, p_n; z) 
 \in \C^{n} \times P_{\alpha_{i_1}}\times\cdots \times P_{\alpha_{i_n}} 
 \times \C_{-\mu_1,\ldots,-\mu_n}$, 
where $(h, \tau_1,\dots, \tau_{n})\in H\times (\C^{\times})^{n}$, 
and where $(b_1,\dots, b_n) \in B^n$. 
By continuity, it is enough to show this when $t_1,\ldots,t_n$ are non-zero,
which we now assume.  We now compute:
\begin{multline*}
 (h, \, \tau_1, \, \dots, \, \tau_{n}) \cdot 
  ((t_1, \, \dots, \, t_{n}; \, p_1, \, \dots, \, p_n; \, z) 
    \cdot (b_1, \, \dots, \, b_n)) \\
 = (h, \, \tau_1, \, \dots, \, \tau_{n}) \cdot  
  (t_1, \, \dots, \, t_{n}; \hfill \\ 
  \, p_1b_1, \, \psi_{t_2}(b_1)\inv p_2 b_2, \,
                             \psi_{t_3}(b_2)\inv p_3 b_3, \,
   \dots, \, \psi_{t_{n}}(b_{n-1})\inv p_n b_n; 
    (b_1,\ldots,b_n)^{-1} \cdot z) \\ 
 = (\tau_1t_1, \, \dots,\,  \tau_{n}t_{n}; \hfill \\
    \, h S(\tau_1) p_1b_1, \,
   S(\tau_2) \psi_{t_2}(b_1)\inv p_2 b_2, \,
   \dots, \,
   S(\tau_{n})\psi_{t_{n}}(b_{n-1})\inv p_n b_n;
    (b_1,\ldots,b_n)^{-1} \cdot z) \\ 
 = (\tau_1t_1, \, \dots, \, \tau_{n}t_{n}; \hfill \\ \,
    h S(\tau_1) p_1b_1, \, 
    (S(\tau_2t_2) b_1 S(\tau_2t_2) \inv)\inv 
                     S(\tau_2) p_2 b_2,  \, \dots, \, 
   (S(\tau_{n}t_{n})b_{n-1} S(\tau_{n}t_{n})\inv)\inv 
                S(\tau_{n})p_n b_n; \\ \hfill (b_1,\ldots,b_n)^{-1} \cdot z) \\ 
 = (\tau_1t_1, \, \dots, \, \tau_{n}t_{n}; \,
    h S(\tau_1) p_1, \, S(\tau_2)p_2, \, \dots, \,
          S(\tau_{n}) p_n; z) \cdot (b_1,\dots,b_n) \\
 = \left((h, \, \tau_1, \, \dots, \, \tau_{n}) 
       \cdot (t_1, \, \dots, \, t_{n}; \, p_1, \, \dots, \, p_n; \, z)\right) 
         \cdot (b_1, \, \dots, \, b_n).
\end{multline*}
The $B^n$ actions being free and proper can be proved by induction on $n$.
\end{proof}

\begin{Construction}[Bott-Samelson family] \labell{construction}
We define $\frakX_{\alpha_{i_1},\ldots,\alpha_{i_n}}$ 
and $\frakL_{\mu_1,\ldots,\mu_n}$ as the quotient spaces 
of $\C^{n}\times P_{\alpha_{i_1}}\times\cdots\times P_{\alpha_{i_n}}
 \times \C_{-\mu_1,\ldots,-\mu_n} $ 
and $\C^{n}\times P_{\alpha_{i_1}}\times\cdots\times P_{\alpha_{i_n}}$
by their free and proper right $B^n$ actions.
Then $\frakL_{\mu_1,\ldots,\mu_n}$ is a holomorphic line bundle over $\frakX_{\alpha_{i_1},\ldots,\alpha_{i_n}}$.
Lemma \ref{lem:commuting actions} implies that
the left $H\times (\C^{\times})^n$-action descends to $\frakX_{\alpha_{i_1},\ldots,\alpha_{i_n}}$ and $\frakL_{\mu_1,\ldots,\mu_n}$.
This yields an $H$-equivariant family over $\C^n$ with an $H$-action
in the sense of Section~\ref{sec:filtrations}.
\end{Construction}

\begin{Remark} \labell{compare with Lu}
Jianghua Lu, with her student Jun Peng \cite{lu-slides, JunPeng}, 
introduced families of deformations of  Bott-Samelson manifolds that are similar to ours. They define the $B^n$ action using $n$ one-parameter subgroups $S_1,\ldots,S_n$; we define it using one one-parameter subgroup $S_1= \ldots = S_n = S$. Here, $n$ is the dimension of the Bott-Samelson manifold.
\end{Remark}

\begin{Remark} \labell{different S}
If $c_1,\ldots,c_n$ are elements of the centre of $G$
(in particular they are contained in $H$) then, in $\frakL_{\mu_1,\ldots,\mu_n}$, we have
\begin{equation*} 
   [t_1, \, \ldots, \, t_n;\, c_1p_1, \, \ldots, \, c_np_n; \, z]
 = [t_1, \, \ldots, \, t_n;\, p_1, \, \ldots, \, p_n; \, 
    {c_1}^{\mu_1} (c_1c_2)^{\mu_2} \cdots (c_1 \cdots c_n)^{\mu_n} z] .
\end{equation*}
This equality, together with the second part 
of Lemma~\ref{one-parameter subgroup properties},
has the following consequences.
If we replace the one-parameter subgroup $S$
by another one-parameter subgroup $\wh{S}$
that has the same properties as $S$ with the same positive integers $q_i$, 
then the resulting spaces $ \frakX_{\alpha_{i_1},\ldots,\alpha_{i_n}}$ and $\frakL_{\mu_1,\ldots,\mu_n}$ are the same as for $S$,
and so is the left action of $H \times (\Cx)^n$ on $ \frakX_{\alpha_{i_1},\ldots,\alpha_{i_n}}$.
But the 
liftings of this left action to 
$\frakL_{\alpha_{i_1},\ldots,\alpha_{i_n}}$
that are obtained from $S$ and from $\wh{S}$ do not coincide. 
If $c \colon \Cx \to H$ is the one-parameter subgroup of the centre of $G$
such that for all $\tau \in \Cx$ we have $\wh{S}(\tau) = c (\tau) S(\tau)$,
then, 
for any element $(h,\, \tau_1,\ \ldots,\, \tau_n)$ of $H \times (\Cx)^n$,
its left action on $\frakL_{\mu_1,\ldots,\mu_n}$ that is obtained from $\wh{S}$
is equal to its left action on $\frakL_{\mu_1,\ldots,\mu_n}$ that is obtained from $S$
times fibrewise multiplication by the scalar
$c(\tau_1)^{\mu_1} \, c(\tau_1\tau_2)^{\mu_2} \, \ldots \,
 c(\tau_1\cdots\tau_n)^{\mu_n}$.
That is, the two left actions differ by a fibrewise $(\Cx)^n$ action
with weight $\vec{m} := (m_1,\ldots,m_n) \in \Z^n$
where $m_j = \left< c , \mu_j+\ldots+\mu_n \right>$. 
The filtrations $\{ F_{\vell} \}$ and $\{ \wh{F}_{\vell} \}$ 
that are obtained from the two left actions
then differ by a shift: $F_{\vell} = \wh{F}_{\vell+\vec{m}}$.
\end{Remark}

We now show that the family $\frakX_{\alpha_{i_1},\ldots,\alpha_{i_n}}$
satisfies the mild technical condition~\eqref{mild technical}:

\begin{lemma}\labell{technical assumption holds} 
For every $j \in \{ 1,\ldots,n \}$,
the action of the $j$th factor $(\Cx)_{\text{$j$th}}$ of $(\Cx)^n$
on the special fibre $X_{\bfzero}$ 
has a fixed point at which all the isotropy weights are non-positive.
\end{lemma}

{
\begin{proof}
By the property (1) in Theorem ~\ref{extension of toric action}, 
it is enough to show
that the isotropy weights for the $(\Cx)_{\text{$j$th}}$-action \eqref{jth}
on $X_{\alpha_{i_1},\ldots,\alpha_{i_n}}$
at the fixed point ${\bf e}=[e,\dots,e]$ are all non-positive.

We will compute the isotropy weights
by obtaining an invariant filtration (a flag structure) 
of $X_{\alpha_{i_1},\ldots,\alpha_{i_n}}$ through $\bf e$;
we'll then take the weights by which $(\Cx)_{\text{$j$th}}$ acts
on the (one dimensional) quotients of successive tangent spaces
through ${\bf e}$.

We have the inclusion maps

\begin{multline*}
\begin{array}{cccccc}
\{{\bf e}\} = X_{\emptyset} & \hookrightarrow 
 & X_{\alpha_{i_n}} & \hookrightarrow 
 & X_{\alpha_{i_{n-1}}, \alpha_{i_n}} & \hookrightarrow \cdots  \\
 & & & {\scriptstyle{[p] \mapsto [e, p]}} & & 
\end{array} \\
\begin{array}{cccc}
 \ldots \hookrightarrow 
 & X_{\alpha_{i_{2}},\dots, \alpha_{i_n}} &\hookrightarrow
 & X_{\alpha_{i_{1}}, \dots,\alpha_{i_n}}. \\
 & & \scriptstyle{[p_2,\dots, p_n] \mapsto [e, p_2,\dots, p_n] } & 
\end{array}
\end{multline*}

Each of these inclusion maps is $(\Cx)_{\text{$j$th}}$-equivariant,
where $\tau \in (\Cx)_{\text{$j$th}}$ acts by 
$$ \tau \cdot [p_k,\ldots,p_n] =  \begin{cases}
 [p_k,\, \ldots,\, p_{j-1},\, S(\tau)p_j,\, p_{j+1},\, \ldots,\, p_n] 
 & \text{ if $j > k$ } \\ 
 [S(\tau) p_k,\, p_{k+1},\, \ldots,\, p_n]
 & \text{ if $j \leq k$ }. 
\end{cases}$$

Also, each consecutive pair is an inclusion of a fibre in a fibration, 
\[ \xymatrix{X_{\alpha_{i_{k+1}}, \dots,\alpha_{i_n}} \ar[r]  
   & X_{\alpha_{i_{k}}, \dots,\alpha_{i_n}}\ar[d]  \\
   & {\phantom{= P_{\alpha_{i_k}}/B}} X_{\alpha_{i_k}} = P_{\alpha_{i_k}}/B,
        }
 \]
where we set $X_{\alpha_{i_{k+1}},\ldots,\alpha_{i_n}} = \{ {\bf e} \}$ 
if $k=n$.  
Moreover the fibration maps are equivariant 
where $\tau \in (\Cx)_{\text{$j$th}}$
acts on $P_{\alpha_{i_k}}/B$ 
by left multiplication by $S(\tau)$ if $j \leq k$
and trivially if $j > k$. 

Passing to the tangent spaces, we get a flag
of $(\Cx)_{\text{$j$th}}$ invariant subspaces
$$ \{ 0 \} \, \subset \, T_{\bf e} X_{\alpha_{i_n}}
           \, \subset \, T_{\bf e} X_{\alpha_{i_{n-1}},\alpha_{i_n}}
 \, \subset \, \ldots \, \subset \, 
 T_{\bf e} X_{\alpha_{i_1},\ldots,\alpha_{i_n}},$$
with consecutive quotients
$$ T_{\bf e} X_{\alpha_{i_k},\ldots,\alpha_{i_n}} / 
   T_{\bf e} X_{\alpha_{i_k+1},\ldots,\alpha_{i_n}} 
   \cong  T_{[e]} ( P_{\alpha_{i_k}}/B ). $$ 

Since $\tau \in (\Cx)_{\text{$j$th}}$ acts on $P_{\alpha_{i_k}}/B$
by left multiplication by $S(\tau)$ if $j \leq k$
and trivially if $j > k$,
by the third part of Lemma~\ref{one-parameter subgroup properties},
the isotropy weight for the $(\Cx)_{\text{$j$th}}$ action 
on the tangent space of $P_{\alpha_{i_k}}/B$ at $[e]$ 
is $-q_{i_k}$ if $j \leq k$ and $0$ if $j > k$,
where $q_i$ are the positive integers involved in the choice of $S$.

As $k$ runs from $1$ to $n$, we obtain that the isotropy weights
for the $(\Cx)_{\text{$j$th}}$ action 
on the tangent space at $\bf e$ of $X_{\alpha_{i_1},\ldots,\alpha_{i_n}}$
are $\underbrace{0,0,\ldots,0}_{j-1}$,$-q_{i_j}$,$\ldots$,$-q_{i_n}$,
which are non-positive, as required.
\end{proof}
}

\begin{proof}[Proof of Theorem \ref{extension of toric action}]
We take the Bott-Samelson family as in Construction~\ref{actions}.
The first property \eqref{intertwining} follows from the definitions
of the relevant spaces  in \eqref{line bundles} and in Construction
\ref{construction}.
The second property \eqref{mild tech} is proved in Lemma \ref{technical
assumption holds}.
For the third property \eqref{unique},
choose the one-parameter subgroup $S \colon \Cx \to H$ 
in Lemma~\ref{one-parameter subgroup existence}
to correspond to the smallest possible positive integer $q$.
If $G$ is semisimple, then the centre of $G$ is discrete,
and by the second part of Lemma~\ref{one-parameter subgroup properties},
the positive integer $q$ determines
the one-parameter subgroup $S \colon \Cx \to H$ uniquely.
If $G$ is reductive but not semisimple, 
then $q$ does not determine $S$ uniquely,
but by Remark~\ref{different S} the equivariant families that correspond
to different choices of $S$ differ by fibrewise tensoring 
with some one dimensional representation of $(\Cx)^n$.
\end{proof}

\section{Bott Canonical Basis?}
\labell{sec:canonical-basis}
The Bott-Samelson equivariant families that we constructed 
in Section \ref{sec:Bott-Samelson}
allow us to apply our results from Section~\ref{sec:filtrations}
to representations of Lie groups. 
By the results of Section~\ref{sec:filtrations}, 
under a ``vanishing of higher cohomology'' assumption, 
this construction gives canonical bases (in an appropriate sense)
for spaces of sections
of holomorphic line bundles over Bott-Samelson manifolds.
Since such spaces 
provide geometric models for irreducible representations of compact Lie groups
(with the restriction of the group action to the maximal torus),
this yields canonical bases for such representations,
under the conjectural vanishing of higher cohomology
for the relevant equivariant families.

As before, let $K$ be a compact Lie group, $G$ its complexification,
$T$ a maximal torus, $H$ the Cartan subgroup, $B$ a Borel subgroup,
and $\{\alpha_1 , \ldots  \alpha_r\}$ the simple positive roots.

\begin{Construction} \labell{con:filtration}
Let $\alpha_{i_1},\ldots,\alpha_{i_n}$ 
be a sequence of simple positive roots;
let $Z_{\alpha_{i_1},\dots, \alpha_{i_n}}$ be the corresponding
complex Bott-Samelson manifold.
Let $\mu_1,\ldots,\mu_n$ be a sequence of weights;
let
$L_{\mu_1,\ldots,\mu_n} \to Z_{\alpha_{i_1},\dots, \alpha_{i_n}}$
be the corresponding holomorphic line bundle.
Choose a one-parameter subgroup of the Cartan subgroup,
$S \colon \Cx \to H$,
that satisfies the conditions in
Lemma~\ref{one-parameter subgroup existence}
with the smallest possible positive integer $q$.
These choices determine
a filtration of the space of holomorphic sections 
$$ \Gamma_\hol(Z_{\alpha_{i_1},\dots, \alpha_{i_n}}, 
                    L_{\mu_1,\ldots,\mu_n})$$
into subspaces $F_{\vell}$, parametrized by vectors $\vell \in \Z^n$.
If the higher cohomologies of the sheaf of holomorphic
sections of the line bundle $\frakL_{\mu_1,\ldots,\mu_n}  \to
\frakX_{\alpha_{i_1},\ldots,\alpha_{i_n} }$  vanish,
then each quotient $F_{\vell}/F_{>\vell}$ is either zero or one dimensional, 
and the direct sum $\oplus_{\vell} F_{\vell}/F_{>\vell}$ 
is isomorphic to 
$ \Gamma_\hol(Z_{\alpha_{i_1},\dots, \alpha_{i_n}}, L_{\mu_1,\ldots,\mu_n})$.
A different choice of $S$ with the same $q$ 
results in the same filtration with a shift of its indices.

Here are the details.

Let $\frakL_{\mu_1,\ldots,\mu_n} 
     \to \frakX_{\alpha_{i_1},\ldots,\alpha_{i_n} } \to \C^n$ 
be the equivariant family constructed in Section \ref{sec:Bott-Samelson}.
%
By Theorem~\ref{t:main} and Theorem~\ref{extension of toric action},
we obtain a filtration of the space of holomorphic sections 
$ \Gamma_\hol(Z_{\alpha_{i_1},\dots, \alpha_{i_n}}, 
                    L_{\mu_1,\ldots,\mu_n})$
into subspaces $F_{\vell}$, parametrized by vectors $\vell \in \Z^n$,
and, for each $\vell$, a ``sweep, twist, extend, restrict'' map 
from $F_{\vell}$ to the space
$\Gamma_\hol(X_{\alpha_{i_1},\ldots,\alpha_{i_n}},
             L_{\mu_1,\ldots,\mu_n}^X)$
of holomorphic sections of the holomorphic line bundle 
$L_{\mu_1,\ldots,\mu_n}^X \to X_{\alpha_{i_1},\dots, \alpha_{i_n}}$
over the Bott tower $X_{\alpha_{i_1},\dots, \alpha_{i_n}}$.
This map is $H$ equivariant, 
is trivial on the subspace $F_{>\vell}$ of $F_{\vell}$,
and its image is in the $\vell^{\text{th}}$ weight space
$\Gamma_\hol(X_{\alpha_{i_1},\ldots,\alpha_{i_n}},
             L_{\mu_1,\ldots,\mu_n}^X)_{\vell}$,
which is at most one-dimensional 
because the Bott tower $X_{\alpha_{i_1},\ldots,\alpha_{i_n}}$ 
is a complex toric manifold.
By Theorem~\ref{t:main},
for each $\vell$, we get an $H$ equivariant map
\begin{equation} \labell{map to special} 
F_{\vell}/F_{>\vell}
\to \Gamma_\hol(X_{\alpha_{i_1},\ldots,\alpha_{i_n}}, 
             L_{\mu_1,\ldots,\mu_n}^X)_{\vell} 
\end{equation}
to a vector space of dimension $\leq 1$.
If the higher cohomology of the sheaf of holomorphic sections
of the line bundle $\frakL_{\mu_1,\ldots,\mu_n}$
vanishes, then the maps~\eqref{map to special}
are one-to-one, and 
a choice of complements of $F_{>\vell}$ in $F_{\vell}$ for each $\vell$
determines an isomorphism of 
$\Gamma_\hol(Z_{\alpha_{i_1},\ldots,\alpha_{i_n}}, 
             L_{\mu_1,\ldots,\mu_n}) $
with the associated graded space $\oplus_{\vell} F_{\vell}/F_{>\vell}$.
In this situation, we obtain a decomposition of 
$\Gamma_\hol(Z_{\alpha_{i_1},\ldots,{\alpha_{i_n}}}, 
             L_{\mu_1,\ldots,\mu_n}) $ into one dimensional spaces.
             
\end{Construction}

\begin{Construction}\labell{con:canonical-basis}
Let $V_\lambda$ be a unitary representation of $K$
of maximal weight $\lambda$. 
Let $\alpha_{i_1},\ldots,\alpha_{i_n}$ be a sequence of simple positive roots 
that corresponds to a reduced expression for the longest element
of the Weyl group.
Applying the previous Construction \ref{con:filtration} 
with $\mu_1=\ldots=\mu_{n-1}=0$ and $\mu_n=\lambda$,
and under the higher-cohomology-vanishing condition, 
we obtain a decomposition of $V_\lambda$ into $H$-invariant
one dimensional subspaces. 
This decomposition depends only on our choices of a Borel subgroup of $G$ 
and a reduced expression for the longest element of the Weyl group.

Here are the details.
Consider the $G$-equivariant holomorphic line bundle 
$L_\lambda = G \times_B \C_{-\lambda} \to G/B$.
By the Borel-Weil theorem, there exists a $K$ equivariant linear isomorphism 
\begin{equation} \labell{Borel Weil}
 V_\lambda \to \Gamma_\hol(G/B,L_\lambda)
\end{equation}
from $V_\lambda$ to the space of holomorphic sections of $L_\lambda$.
By Schur's lemma, every two such linear isomorphisms
differ by multiplication by a complex scalar.
We have a pullback diagram
$$ \xymatrix{
 L_{0,\ldots,0,\lambda} \ar[r] \ar[d] & L_\lambda \ar[d] \\
 Z_{\alpha_{i_1},\dots, \alpha_{i_n}} \ar[r] & G/B
}$$
where the horizontal arrows are induced from the multiplication map
${P_{\alpha_{i_1}} \times \ldots \times P_{\alpha_{i_n}} \to G}$.
By Demazure \cite[\S 5 Proposition 5]{demazure}, this diagram induces a $B$-equivariant
linear isomorphism between the spaces of holomorphic sections: 
\begin{equation} \labell{pullback}
 \Gamma_\hol (G/B,L_\lambda) \xrightarrow{\cong}
 \Gamma_\hol (Z_{\alpha_{i_1},\dots, \alpha_{i_n}} ,
                L_{0,\ldots,0,\lambda}).
\end{equation}
Composing with the Borel-Weil isomorphism~\eqref{Borel Weil},
we get a $T$-equivariant linear isomorphism 
\begin{equation} \labell{get an iso}
 V_\lambda \xrightarrow{\cong} 
 \Gamma_\hol (Z_{\alpha_{i_1},\dots, \alpha_{i_n}} ,
                L_{0,\ldots,0,\lambda}).
\end{equation}
Because the Borel-Weil isomorphism~\eqref{Borel Weil}
is unique up to scalar
and the isomorphism~\eqref{pullback} 
is prescribed (it is given by pullback),
the isomorphism~\eqref{get an iso} that we get in this way
is unique up to scalar.
Because the isomorphism~\eqref{get an iso} is unique up to scalar,
our filtration of the space of sections
$\Gamma_\hol (Z_{\alpha_{i_1},\dots, \alpha_{i_n}} ,
                L_{0,\ldots,0,\lambda})$
gives a filtration of the representation $V_\lambda$
that is independent of the choice of the 
Borel-Weil isomorphism~\eqref{Borel Weil}.
We can then view the spaces $F_{\vell}$ as subspaces of $V_\lambda$.
The given $K$ invariant Hermitian inner product on $V_\lambda$ 
determines a choice
of complement of $F_{>\vell}$ in $F_{\vell}$ for each $\vell$,
hence a map $\oplus_{\vell} F_{\vell}/F_{>\vell} \to V_\lambda$.
If the higher cohomology of the sheaf of holomorphic sections
of the line bundle $\frakL_{0,\ldots,0,\lambda} 
 \to \frakX_{\alpha_{i_1},\ldots,\alpha_{i_n}}$ vanishes,
then this map gives a decomposition of $V_\lambda$ into $H$-invariant
one dimensional subspaces.
This decomposition depends only on our initial choices
of a maximal torus $T$ in $K$,
a Borel subgroup $B$ in $G$,
and a reduced expression for the longest element of the Weyl group.

{\it A priori}, 
the construction also depends on a choice of a one-parameter group 
$S \colon \Cx \to H$; 
we choose $S$ as in Lemma~\ref{one-parameter subgroup existence}
with the smallest possible integer $q$.
If $G$ is semisimple, there exists only one such an $S$.
If $G$ is not semisimple, there can be more than one such an $S$,
but different choices for $S$
result in the same filtration, with only a shift in the indexing set.
See Remark~\ref{different S}.
\end{Construction}


\section{Example}
\labell{sec:examples}

In this section, we consider a concrete example. Let $G=\SL(3,\C)$ and $B$ be the subgroup of upper triangular matrices,
with the positive simple roots $\alpha_1=(1,-1,0)$
and $\alpha_2=(0, 1, -1)$.  The reduced decomposition
$w_0=s_{\alpha_1}s_{\alpha_2}s_{\alpha_1}$ of the longest element
$w_0$ of the Weyl group gives rise to the Bott-Samelson manifold
$Z_{\alpha_1,\alpha_2,\alpha_1}$.
We consider the line bundle  $L_{(0, \varpi_{2}, \varpi_{1})}=P_{\alpha_1} \times P_{\alpha_2} 
  \times P_{\alpha_1} \times_{B^3} \C_{(0, -\varpi_{2}, -\varpi_{1})}$ 
over $Z_{\alpha_1,\alpha_2,\alpha_1}$, where $\varpi_1$ and $\varpi_2$ are the fundamental weights (see \S \ref{sec:Bott-Samelson}).
Its space of holomorphic sections
$\Gamma_\hol(Z_{\alpha_1,\alpha_2,\alpha_1}, L_{(0, \varpi_{2}, \varpi_{1})})$ is isomorphic,
as a $B$-representation,
to $V_{\lambda}$, the irreducible representation of $G$ with
the highest weight $\lambda:= \varpi_{1}+\varpi_{2}=\alpha_1+\alpha_2$;
see \cite[Theorem~6]{lakshmibai-littelmann-magyar}.
Taking $L_\lambda := G \times_B \C_{-\lambda}$
(see \S\ref{sec:overview}),
we have a pullback diagram of holomorphic line bundles
$$ \xymatrix{
 L_{(0, \varpi_{2}, \varpi_{1})}&\cong &L_{(0, 0, \lambda)} \ar[rr] \ar[d] && L_{\lambda} \ar[d] \\ 
 &&Z_{ \alpha_1,\alpha_2,\alpha_1} \ar[rr]^{\text{multiplication}} && G/B 
}$$
that induces a bijection on their spaces of holomorphic sections.

For an explicit  computation, we will use a certain trivialization
induced by a Pl\"ucker embedding.  Let $n := 3$.
For each $1 \leq k \leq n$, 
let $\Gr(k)$ be the Grassmannian of $k$-dimensional subspaces of $\C^n$,
and let $\mathbb{F}^k=\mathrm{span}(e_1,\dots,e_k)$ be the $k$-dimensional
subspace of $\C^n$ spanned by the standard basis vectors $e_1,\dots, e_k$.
Let $\St(k)$ denote the Stiefel manifold
of linearly independent $k$-tuples of vectors in $\C^n$,
and consider the Pl\"ucker embedding 
\[ w \colon \Gr(k)\rightarrow \mathbb{P}(\wedge^k\mathbb{C}^n), 
\quad
 \Span\{v_1,\dots,v_k\}\mapsto [v_1\wedge\cdots \wedge v_k] ,\] 
induced from the map 
\[ \St(k)\rightarrow \wedge^k\mathbb{C}^n \ssminus \{ 0 \} , 
\quad
 \big( v_1,\dots,v_k \big) \mapsto v_1\wedge\cdots \wedge v_k \, . \] 
Also consider the embedding 
$$ Z_{\alpha_k} := P_{\alpha_k}/B \to \Gr(k), 
   \quad [p] \mapsto p {\mathbb F}^k $$
of the one-step Bott-Samelson manifold into the Grassmannian,
which is induced from the embedding 
$$ P_{\alpha_k} \to \St(k), \quad p \mapsto (pe_1,\ldots,pe_k) $$
of the parabolic into the Stiefel manifold.

Consider the pullback $w^* L_1$
of the hyperplane bundle $L_1 \to \mathbb{P}(\wedge^k\mathbb{C}^n)$
under the Pl\"ucker embedding.
Its further pullback under the embedding $Z_{\alpha_k} \to \Gr(k)$
is isomorphic to the line bundle $L_{\varpi_k} \to Z_{\alpha_k}$
(see \S\ref{sec:Bott-Samelson}),
and this pullback yields a bijection on the spaces of holomorphic sections 
$$ \Gamma_\hol(\Gr(k), w^*L_1) 
   \xrightarrow{\cong} \Gamma_\hol(Z_{\alpha_k},L_{\varpi_k}) .$$

Pulling back to the Stiefel manifold, we obtain a trivial line bundle.
Through this pullback,
the spaces of holomorphic sections, $\Gamma_\hol(\Gr(k), w^*L_1)$
and $\Gamma_\hol(Z_{\alpha_k},L_{\varpi_k})$,
get identified with the span of the $k$-minors in the Stiefel coordinates.

For example, the space of holomorphic sections 
of the line bundle $L_{\varpi_2} \to Z_{\alpha_2}$
is identified with the linear span of the functions 
$$ \Delta_{ab} \colon \St(2) \to \C \quad , \quad
\Delta_{ab}(y) = \det \left[ 
\begin{array}{cc}
y_{a1} & y_{a2} \\
y_{b1} & y_{b2} \, ,
\end{array}
\right] 
$$
for $a,b \in \{ 1, 2, 3 \}$, where
$$ y = 
\left(
\left[ \begin{array}{c} y_{11} \\ y_{21} \\ y_{31} \end{array} \right] ,
\left[ \begin{array}{c} y_{12} \\ y_{22} \\ y_{32} \end{array} \right]
\right)\in \St(2)\, .$$

These maps are expressed in the following diagram,
in which all the squares are pullback diagrams,
the vertical arrows on the left are principal $B$ bundles,
the vertical arrows in the middle are principal $\GL(k)$ bundles,
the vertical arrows on the right are principal $\Cx$ bundles,
and 
the arrows that point from the back to the front are holomorphic line bundles.

\[
  \xymatrix{
& P_{\alpha_k} \times \C \ar[rr] \ar[ld] \ar'[d][dd] && \St(k) \times \C 
\ar[rr] \ar[dl] \ar'[d][dd]
&& (\wedge^k\C^n\smallsetminus\{0\})\times\C 
\ar[dd]\ar[dl] \\
P_{\alpha_k}\ar[rr] \ar[dd] &&  \St(k) 
 \ar[rr]^(.6){} \ar[dd]^(.68){} 
 && \wedge^k\C^n\smallsetminus\{0\} 
 \ar[dd]^(.68){}& \\
& L_{\varpi_k}\ar'[r][rr] \ar[ld] && w^* L_1
\ar'[r][rr] \ar[dl] && L_1 \ar[dl] \\ 
Z_{\alpha_k}\ar[rr] && \Gr(k) \ar[rr]^(.6){w} 
 &&  \mathbb{P}(\wedge^k\C^n) 
&
}
\]

\bigskip
\bigskip
\bigskip
\bigskip

We also have the embedding of $Z_{\alpha_1, \alpha_2,\alpha_1}$ 
into the product of Grassmannians by successive multiplications (\cite[\S4.1]{lakshmibai-littelmann-magyar}, \cite[\S1]{magyar 98 Schubert}):

\[\mu \colon Z_{\alpha_1,\alpha_2,\alpha_1}
 \rightarrow \Gr(1)\times \Gr(2)\times \Gr(1);  \quad
 [p,q,r]\mapsto (p\mathbb{F}^1,pq\mathbb{F}^2,pqr\mathbb{F}^1) \, , \]
which is 
induced from the embedding of the product of parabolics
into the product of Stiefel manifolds:
\[ \mu' \colon 
P_{\alpha_1} \times P_{\alpha_2} \times P_{\alpha_1}
 \rightarrow \St(1) \times \St(2) \times \St(1); \quad
 (p,q,r) \mapsto ( pe_1; \, pq(e_1,e_2); \, pqre_1 ).\]

In our example, the line bundle $L_{(0, \varpi_{2}, \varpi_{1})}$ 
is isomorphic to the pullback under $\mu$ of the line bundle 
$w^*L_0 \boxtimes w^*L_1 \boxtimes w^*L_1$
over $\Gr(1)\times \Gr(2)\times \Gr(1)$,
where $L_0$ is the trivial bundle 
and $L_1$ is --- as before --- the hyperplane bundle
(see \cite[\S 4.1]{lakshmibai-littelmann-magyar}). 
Thus, we have a pullback map,
from the span of all products of the form 
$\Delta_{abc}(x,y,z) = \Delta_{ab}(y)\Delta_c({z})$, 
where $a,b,c \in \{ 1, 2, 3 \}$ and where $\Delta_{ab}, \Delta_c$,
are the corresponding minors on the Stiefel coordinates $x,y,z$
for $(x,y,z) \in \St(1) \times \St(2) \times \St(1)$,
to the space of holomorphic sections
$\Gamma_\hol(Z_{\alpha_1,\alpha_2,\alpha_1}, L_{(0, \varpi_{2}, \varpi_{1})})$.
This pullback map is in fact one-to-one
\cite[\S 4.1]{lakshmibai-littelmann-magyar}
and onto \cite[\S 4.2]{lakshmibai-littelmann-magyar},
so we can identify this space of holomorphic sections, 
$\Gamma_\hol(Z_{\alpha_1,\alpha_2,\alpha_1}, L_{(0, \varpi_{2}, \varpi_{1})})$, 
with the span of the functions
$$\Delta_{abc} \colon \St(1) \times \St(2) \times \St(1) \to \C .$$
We denote by $s_{(abc)}$ the section that corresponds 
to the function $\Delta_{abc}$ under this identification.

We now have the following diagram, where the squares are pullback diagrams,
the vertical arrows on the left are principal $B^3$ bundles
and the vertical arrows on the right are principal 
$\GL(1) \times \GL(2) \times \GL(1)$ bundles, 
and the arrows that point from the back to the front
are holomorphic line bundles.

\[
  \xymatrix@C=-1em{
 & P_{\alpha_1}\times P_{\alpha_2}\times P_{\alpha_1}\times \C 
\ar[rrr]^{} \ar[dl] \ar'[d][ddd]^{\pr} 
 &&& \St(1)\times \St(2)\times \St(1)\times\C 
\ar[ddd]\ar[dl]_{} \\
 P_{\alpha_1}\times P_{\alpha_2}\times P_{\alpha_1}
 \ar[rrr]^(.55){\mu'} \ar[ddd]^(.68){} 
\ar@/^/[ur]^(.4){\hat s} 
 &&& \St(1)\times \St(2)\times \St(1)
 \ar[ddd]^(.68){}& \\
 &&& \\
&  L_{(0, \varpi_{2}, \varpi_{1})}
 \ar'[rr]^{}[rrr] \ar[dl] 
 &&& w^*L_0 \boxtimes w^*L_1 \boxtimes w^*L_1\ar[dl]_{}
\\ Z_{\alpha_1,\alpha_2,\alpha_1} \ar[rrr]^(.55){\mu} 
\ar@/^/[ur]^{s} 
   &&&
   \Gr(1)\times \Gr(2)\times \Gr(1)&
}
\]

More precisely, the function $\Delta_{abc}$ determines the section
\[ \hat{s}_{} \colon P_{\alpha_1} \times P_{\alpha_2} \times P_{\alpha_1}
  \to P_{\alpha_1} \times P_{\alpha_2} \times P_{\alpha_1} \times \C 
\quad , \quad
   \hat{s}_{} (p,q,r) := (p,q,r,\Delta_{abc}(\mu'(p,q,r)) \, , \]
which descends to the section 
\[  s \colon Z_{\alpha_1,\alpha_2,\alpha_1}\rightarrow L_{(0, \varpi_{2}, \varpi_{1})}
\quad , \quad
    s_{}([p,q,r]) := \pr \circ \hat{s}_{} (p,q,r) .\]
This is well-defined, since 
\begin{multline*}
   \Delta_{abc}(\mu'((p,q,r)\cdot (b_1, b_2,b_3))) 
 = \Delta_{ab}(pqb_2(e_1,e_2)) \Delta_{c}(pqrb_3 e_1 ) \\
 = \Delta_{abc}(\mu'((p,q,r))) b_2^{\varpi_2} b_3^{\varpi_1},
 \quad (b_1,b_2,b_3)\in B^3  \, ,
\end{multline*}
where $\varpi_1$ and $\varpi_2$ are the fundamental weights.
We then define $s_{(abc)} := s$.

Fix the one-parameter subgroup
$$ S(t) \colon \mathbb{C}^{\times}\rightarrow H
 \quad , \quad 
     t \mapsto \left(\begin{array}{ccc}
 t & 0 & 0 \\ 0 & 1 & 0 \\ 0 & 0 & t^{-1} \end{array} \right),$$ 
which satisfies the conditions 
in Lemma~\ref{one-parameter subgroup existence}.
Let $(\Cx)^3$ act on the trivial line bundle
$\C^3 \times P_{\alpha_1} \times P_{\alpha_2} \times P_{\alpha_1} \times \C$ 
by 
\[(t_1,t_2,t_3) \cdot (a_1,a_2,a_3, p, q, r, \theta)
 := (a_1t_1,a_2t_2,a_3t_3, S(t_1)p, S(t_2)q, S(t_3)r, \theta). \] 
As described in Section~\ref{sec:Bott-Samelson},
this action descends to an equivariant family
$\frakL \to \frakX \to \C^3$,
where $\frakL = \frakL_{(0, \varpi_{2}, \varpi_{1})}$ 
and $\frakX = \frakX_{\alpha_1,\alpha_2,\alpha_1}$.
Now we consider the sweeps (equivariant extensions)
$\tilde{\hat{s}}_{}$ and  $\tilde s_{}$ 
of $\hat{s}_{}$ and $s_{}$.
We have the following diagram,
in which the dotted arrows reflect that the sweeps
might not be defined everywhere,
and in which the horizontal arrows are the natural embeddings
as the fibres over $\bfone := (1,1,1) \in \C^3$.
\[
  \xymatrix@C=-3em{
& P_{\alpha_1} \times P_{\alpha_2} \times P_{\alpha_1} \times \C 
 \ar[rr]^{I} \ar[dl] \ar'[d][dd]  && 
 \C^3 \times P_{\alpha_1} \times P_{\alpha_2} \times P_{\alpha_1} \times \C 
\ar[dd]^(.65){\pr} \ar[dl]^{\pi} \\
 P_{\alpha_1} \times P_{\alpha_2} \times P_{\alpha_1}\ar[rr]^(.5){I} 
 \ar@/^/[ur]^(.4){\hat s} \ar[dd] && 
 \C^3 \times P_{\alpha_1} \times P_{\alpha_2} \times P_{\alpha_1}
 \ar[dd]^(.65){\pr} \ar@/^/@{.>}[ur]^(.4){\tilde{\hat{s}}} & \\
&  L_{(0, \varpi_{2}, \varpi_{1})}\ar'[r]^{I}[rr] \ar[dl] 
 && \frakL_{(0, \varpi_{2}, \varpi_{1})} \ar[dl]^{\pi} \\ 
Z_{\alpha_1,\alpha_2,\alpha_1} \ar[rr]^{I} \ar@/^/[ur]^{s }&&    
\frakX_{\alpha_1,\alpha_2,\alpha_1}\ar@/^/@{.>}[ur]^{\tilde s}  &
}
\]
Because
\[ \tilde{s}([a_1,a_2,a_3,p,q,r])
 = \pr\circ \tilde{\hat{s}}(a_1,a_2,a_3,p,q,r) \quad \text{and} \quad 
 (\bft^{-\ell}\cdot \tilde{s})\circ \pr 
 = \pr\circ (\bft^{-\ell}\cdot \tilde{\hat{s}}),\]
we have that 
\[ \bft^{-\ell}\cdot\tilde s \ \text{ extends to }  \
 \frakX_{\alpha_1,\alpha_2,\alpha_1} \ \text{ if and only if }  \
\bft^{-\ell}\cdot\tilde{\hat{s}} \ \text{ extends to } \ 
\C^3\times P_{\alpha_1}\times P_{\alpha_2}\times P_{\alpha_1}.\]

Now,
\[ \tilde{\hat{s}}(a_1,a_2,a_3, p,q,r)
 = (a_1,a_2,a_3) \cdot \hat{s}((a_1,a_2,a_3)\inv\cdot(a_1,a_2,a_3, p,q,r)) \]
\[ = (a_1,a_2,a_3) \cdot \hat{s} (S(a_1)\inv p, S(a_2)\inv q, S(a_3)\inv r) \]
\[ = (a_1, a_2, a_3) \cdot 
 \left( 1,1,1,S(a_1)\inv p, S(a_2)\inv q, S(a_3)\inv r,
 \Delta_{abc} \left(\mu'(S(a_1)\inv p, S(a_2)\inv q, S(a_3)\inv r ) \right)\right) \]
\[= \left( a_1,a_2,a_3, p,q,r, 
 \Delta_{abc}(\mu'(S(a_1)\inv p, S(a_2)\inv q, S(a_3)\inv r)) \right) .\]

Therefore, we can identify $\tilde{\hat{s}}(a_1,a_2,a_3, p,q,r)$ 
with the complex valued function 
\[ \Delta_{abc}\left(\mu'(S(a_1)\inv p, S(a_2)\inv q, S(a_3)\inv r )\right)\]

The sequence of indices $abc$ indexing the function $\Delta_{abc}$ 
is called a \emph{tableau}. In  \cite{lakshmibai-magyar:GLn} 
the authors introduced the notion of \emph{standard tableaux} 
and proved that the set of standard tableaux gives rise 
to a basis of $\Gamma_\hol(Z_{\alpha_1,\alpha_2,\alpha_1}, L_{(0, \varpi_{2}, \varpi_{1})})$. 

In our example, the standard tableaux are 
\[(121),\ (122),\ (131),\ (231),\ (232),\ (132),\ (133),\ (233).\] 

We will now compute the $\ell$-vector (see Definition~\ref{l-vector}) 
for the section corresponding to the standard tableau $(abc)=(121)$.
Let 
\[ p = \left(\begin{array}{ccc}
 p_{11} & p_{12} & p_{13} \\ 
 p_{21} & p_{22} & p_{23} \\
 0 & 0 & p_{33} 
\end{array}\right),  \quad 
   q = \left( \begin{array}{ccc}
 q_{11} & q_{12} & q_{13} \\
 0 & q_{22} & q_{23} \\
 0 & q_{32} & q_{33} 
\end{array}\right), \quad 
   r = \left(\begin{array}{ccc}
 r_{11} & r_{12} & r_{13} \\
 r_{21} & r_{22} & r_{23} \\
 0 & 0 & r_{33}
\end{array}\right). \]
  
Then, 
\begin{multline*}
\Delta_{121}(\mu'(p,q,r))
  =  \Delta_{12}(pq(e_1, e_2))\Delta_{1}(pqre_1) \\
  =  \left( p_{11}q_{11} 
       (p_{21}q_{12} + p_{22}q_{22}+p_{23}q_{32}) 
   - (p_{11}q_{12}+p_{12}q_{22}+p_{13}q_{32}) p_{21}q_{11} \right) \\
    \cdot \left( 
     p_{11}q_{11}r_{11} + (p_{11}q_{12}+p_{12}q_{22}+p_{13}q_{32}) r_{21} 
    \right)\,.
\end{multline*}

Substituting, we get
\begin{multline*}
\Delta_{121}\left(\mu'(S(t_1)^{-1}p, S(t_2)^{-1}q, S(t_3)^{-1}r)\right)  \\
 = 
\left( \dfrac{p_{11}q_{11}}{t_1t_2}
 \left( \dfrac{p_{21}q_{12}}{t_2}+p_{22}q_{22}+t_2p_{23}q_{32} \right)
 - \left(\dfrac{p_{11}q_{12}}{t_1t_2}+\dfrac{p_{12}q_{22}}{t_1}+\dfrac{t_2p_{13}q_{32}}{t_1} \right) \dfrac{p_{21}q_{11}}{t_2} \right)  \\
 \cdot \left( \dfrac{p_{11}q_{11}r_{11}}{t_1t_2t_3}+(\dfrac{p_{11}q_{12}}{t_1t_2}+\dfrac{p_{12}q_{22}}{t_1}+\dfrac{t_2p_{13}q_{32}}{t_1})r_{21} \right) \\
 = t_1^{-2}t_2^{-2}t_3^{-1}
 \left( q_{11}q_{22}(p_{11}p_{22} - p_{12}p_{21}) +t_2q_{11}q_{32}(p_{11}p_{23} - p_{13}p_{21}) \right) \\
\cdot
 \left( p_{11}q_{11}r_{11}+(t_3p_{11}q_{12}+t_2t_3p_{12}q_{22}+t_2^2t_3p_{13}q_{32})r_{21} \right) ,
\end{multline*}
where $S(t) \colon \mathbb{C}^{\times}\rightarrow H$
is given by $t \mapsto \left(\begin{array}{ccc}
 t & 0 & 0 \\ 0 & 1 & 0 \\ 0 & 0 & t^{-1} \end{array} \right)$. 

This implies that the $\ell$-vector for the section $s_{(121)}$
that corresponds to the function $\Delta_{121}$ is equal to $(-2,-2,-1)$.

A similar calculation gives rise to the $\ell$-vectors
for the other standard tableaux. 
We list these in the following table:

\def\arraystretch{1.5}
 \[\begin{array}{c|c|c} 
     \text{ standard tableau } & s & \vec{\ell_s} \\
      \hline
      (121) & s_{(121)} &  (-2,-2,-1)\\
      (122) & s_{(122)} & (-1,-2,-1) \\
      (131) & s_{(131)} & (-1,-1,-1) \\
      (231) & s_{(231)} & (0,-1,-1) \\
      (232) & s_{(232)} & (1,-1,-1)\\
      (132) & s_{(132)} & (0,-1,-1)\\
      (133) & s_{(133)} & (1,1,0)\\
      (233) & s_{(233)} & (2,1,0)\\
        \hline
            & s_{(231)} - s_{(132)} &  (0,0,0)\\
 \end{array}
 \]
\def\arraystretch{1}

The sections $s_{(231)}$ and $s_{(132)}$ corresponding
to the standard tableaux $(231)$ and $(132)$, respectively, have
the same $\ell$-vector, $(0,-1,-1)$.
This means that these sections
belong to the same filtered piece $F_{(0,-1,-1)}$
and are equivalent in the quotient space $F_{(0,-1,-1)}/F_{>(0,-1,-1)}$.  
Their difference $s_{(231)}-s_{(132)}$ has
the $\ell$-vector $\vec{\ell_s}=(0,0,0)>(0,-1,-1)$. 
This implies that 
the space $\Gamma_\hol(Z_{\alpha_1,\alpha_2,\alpha_1},
L_{(0, \varpi_{2}, \varpi_{1})})$ decomposes into eight one-dimensional spaces
$F_{\vec{\ell}}/F_{>\vec{\ell}}$, where $\vec{\ell}$ are the eight
distinct vectors written on the second column of the above table.

\begin{Remark}  \labell{string polytope}
In the example in this section 
(with $G=\SL(3,\C)$, $w_0=s_{\alpha_1}s_{\alpha_2}s_{\alpha_1}$, 
and $L_{(0, \varpi_{2}, \varpi_{1})}$), 
the convex hull of the eight $\ell$-vectors is 
upper-triangularly and unimodularly equivalent to the corresponding
string polytope,  
which was realized as a Newton-Okounkov body in 
\cite[\S 5]{kaveh}.
More precisely, by the following affine linear transformation
of $\Z^3$,
\[ x \ \mapsto \ Ax+b \ ,
   \quad \text{ where } \quad 
A = \left(\begin{array}{ccc} 
 1 & 1 & 1\\ 0 & 1 & 1\\ 0& 0& 1\end{array}\right)
\quad \text{ and } \quad b = 
\left(\begin{array}{c} -2\\ -2\\ -1\end{array}\right),\] 
the string polytope is transformed to the convex hull 
of these eight $\ell$-vectors.
It would be interesting to find a connection between our work 
and Newton-Okounkov body theory that explains this observation.
For example, perhaps the notion of the $\ell$-vector $\vec{\ell_s}$ 
of a section $s$ 
can be extended to a valuation whose Newton-Okounkov body 
is equal to the convex hull of the collection 
of the $\ell$-vectors~$\vec{\ell_s}$. See Remark \ref{rk:valuations}.
\end{Remark}


\appendix
\section{A ``holomorphic Hadamard lemma''}
\labell{app:hadamard}

The purpose of this appendix is to prove Theorem~\ref{HA},
which we use in the proof of Lemma~\ref{l:hadamard}. Throughout this section,  ``sheaf'' refers to a presheaf that satisfies the sheaf conditions, and ``cohomology'' refers to  \v{C}ech cohomology as in Bott-Tu (\cite[\S 10]{bott-tu}).
Fix a natural number $r$,
a complex manifold $\frakX$, a submersion 
$$\pi = (t_1,\ldots,t_r) \colon \frakX \to \C^r,$$
and a holomorphic line bundle 
$$\frakL \to \frakX.$$
Denote by $\calO_{\frakL}$ the sheaf of holomorphic sections 
of $\frakL$:
\ $ \calO_\frakL(U) = \Gamma_{\hol}(U,\frakL)$.
Our goal is to prove the following theorem.

\begin{Theorem}\labell{HA}
Fix a subset $A \subseteq \{1,\ldots,r\}$.
For any open subset $U \subseteq \frakX$
such that $H^{\geq 1}(U,\calO_{\frakL}) {= \{ 0 \}}$,
and for any section $s \in \Gamma_\hol(U,\frakL)$,
if $s$ vanishes on $U \cap \bigcap\limits_{i \in A} \{ t_i=0 \}$,
then there exist sections $s_i \in \Gamma_\hol(U,\frakL)$, 
indexed by $i \in A$, such that $s = \sum_{i \in A} t_i s_i$. 
\end{Theorem}

\begin{Remark}
If $A = \emptyset$, then the conclusion of Theorem~\ref{HA}
is trivially true:
if $s$ vanishes on $U \cap \left( \text{an empty intersection} \right)$,
which is $U$, then $s$ is the empty sum, which is zero. 
\end{Remark}

\begin{Remark}
As we learned from Michel Brion, this theorem is known to experts,
and it can be proved using an appropriate Koszul complex of modules.
Our proof is direct.  We expect our proof to be in some sense equivalent
to the one that would be obtained from the Koszul machinery;
it would be interesting to spell this out.
\end{Remark}

The proof of Theorem~\ref{HA} involves additional sheaves,
which we now introduce.

For any subset $A$ of $\{1,\ldots,r\}$, 
denote by $|A|$ the number of elements of $A$, 
and for any integer $k$ such that $0 \leq k \leq |A|$, 
denote by $\binom{A}{k}$ the set of subsets of $A$ that have $k$ elements.
For any such $A$ and $k$,
consider the sheaf on $\frakX$ that is given by\[ 
 \calS_{A,k} (U) := \left\{ 
 \left( s_{j,B'} \right) \in 
   \textstyle{\bigoplus\limits_{\substack{ (j,B') \\
 j \in A , \, B' \in \binom{A \ssminus \{ j \}}{k-1} }}}
\Gamma_\hol(U,\frakL)
 \ \left| \ {\textstyle{ 
\sum\limits_{ B \in \binom{A}{k} }
 {\prod\limits_{\ell \in B}} t_\ell 
 \sum\limits_{j \in B} s_{j,B \ssminus \{ j \} } = 0 }} \right. \right\} .
\]

For later reference we note that, 
when $k=0$, we have $\calS_{A,0}(U) = \{ 0 \}$,
and when $k=1$, 
\[
\calS_{A',1}(U) = \left\{ (s_j) \in \bigoplus_{j \in A'}
 \Gamma_\hol(U,\frakL) \ \left| \ \sum_{i \in A'} t_i s_i = 0 \right. \right\}.
\]

\begin{Theorem}\labell{VAk}
Fix a subset $A \subseteq \{ 1 , \ldots , r \}$
and an integer $k$ such that $1 \leq k \leq |A|$.
Then, for any open subset $U \subseteq \frakX$,
if $H^{\geq 1}(U;\calO_{\frakL}) = \{ 0 \}$,
then $H^{\geq 1}(U;\calS_{A,k}) = \{ 0 \}$. 
\end{Theorem}

\begin{Theorem}\labell{DAk}
Fix a subset $A \subseteq \{ 1,\ldots,r \}$
and an integer $k$ such that $1 \leq k \leq |A|$.
Then, for any open subset $U \subseteq \frakX$
such that $H^{\geq 1}(U;\calO_{\frakL}) = \{ 0 \}$,
and for any
collection of sections $s_{j,B'} \in \Gamma_\hol(U,\frakL)$,
indexed by $j \in A$ and $B' \in \binom{A \ssminus \{ j \}}{k-1}$,
such that $(s_{j,B'}) \in \calS_{A,k}(U)$,
there exists a collection of sections $s_{i,B} \in \Gamma_\hol(U,\frakL)$,
indexed by $i \in A$ and $B \in \binom{ A \ssminus \{ i \} }{k}$,
such that, for every $B \in \binom{A}{k}$,
\begin{equation} \labell{star55}
 \sum_{j \in B} s_{ j , B \ssminus \{ j \} } 
 = \sum_{i \in A \ssminus B} t_i s_{i,B}. 
\end{equation}
\end{Theorem}

\begin{Remark} \labell{DAk-remark}
The collection $\left( s_{i,B} \right)$ 
that is obtained in Theorem~\ref{DAk}
is necessarily in $\calS_{A,k+1}(U)$.  Indeed,
\[
 \sum\limits_{\wt{B} \in \binom{A}{k+1}}
 \prod\limits_{\ell \in \wt{B}} t_\ell 
 \sum\limits_{j \in \wt{B}}
 s_{j, \, \wt{B} \ssminus \{ j \} }  
 = \sum\limits_{B \in \binom{A}{k}} 
 \prod_{\ell \in B} t_\ell 
 \sum\limits_{j \in A \ssminus B} 
 t_j s_{j,B} \\ 
 = \sum\limits_{ B \in \binom{A}{k} } 
 \prod\limits_{\ell \in B} t_\ell 
 \sum\limits_{j \in B}
 s_{j , B \ssminus \{ j \} } 
 = 0
\]
where the first equality is by setting $B = \wt{B} \ssminus \{ j \}$, 
the second equality is by~\eqref{star55},
and the third equality is because
 $\left( s_{j,B'} \right) \in \calS_{A,k} (U)$.
\end{Remark}

\begin{Lemma}[\textbf{Base cases of Theorems~\ref{DAk} and~\ref{VAk}}]
\labell{DAk-VAk-base}
Fix a subset $A \subseteq \{1,\ldots,r\}$, and let $k=|A|$.  
Then the conclusions of Theorems~\ref{DAk} and~\ref{VAk} are true.
\end{Lemma}

\begin{proof}
We claim that
\begin{equation} \labell{product}
\calS_{A,k} (U) = \left\{ \textstyle{
\left.
(s_{j,A \ssminus \{ j \} }) \in \bigoplus\limits_{j \in A} \Gamma_\hol(U,\frakL)
 \ \right| \ \sum\limits_{j \in A} s_{j,A \ssminus \{ j \}} = 0 } \right\} . 
\end{equation}
Indeed, since $k=|A|$, the sets $B'$ and $B$ that occur
in the definition of $\calS_{A,k}$
are $B' = A \ssminus \{ j \}$ and $B=A$.
So the definition of $\calS_{A,k}(U)$ simplifies to
\[ \textstyle{ \calS_{A,k}(U) = \left\{ \left.
   (s_{j, A \ssminus \{ j \} }) \in 
    \bigoplus\limits_{j \in A} \Gamma_{\hol}(U,\frakL) \ \right| \ 
  \prod\limits_{\ell \in A} t_\ell 
  \sum\limits_{j \in A} s_{j , A \ssminus \{ j \} } = 0 \right\} . } \]
Because the product of the $t_\ell$s is non-zero on an open dense set,
the condition on $(s_{j,A \ssminus \{ j \} })$ further simplifies to 
the condition $\sum\limits_{j \in A} s_{j , A \ssminus \{ j \} } = 0 $,
giving~\eqref{product}.

The right hand side of~\eqref{star55} is (the empty sum, hence)
the zero section.
By~\eqref{product}, the left hand side of~\eqref{star55} 
is also the zero section.  
This proves the conclusion of Theorem~\ref{DAk} in this case.

By~\eqref{product}, and writing $s_j := s_{j , A \ssminus \{ j \}}$
to simplify notation, we have
\[ \calS_{A,k} (U) = \left\{ \left.
\textstyle{ (s_j) \in \bigoplus\limits_{j \in A} \Gamma_\hol(U,\frakL) }
 \ \right| \ \textstyle{ \sum\limits_{j \in A} s_j = 0 } \right\} .\]

Let $i = \min A$.
Then $\calS_{A,k}(\cdot)$ is isomorphic to 
$\bigoplus\limits_{j \in A \ssminus \{ i \} } \Gamma_{\hol}(\cdot,\frakL) $
by the projection
\[ 
 (s_j)_{j \in A} \mapsto (s_{j})_{j \in A \ssminus \{ i \} } .
\] 
Indeed, $s_i$ can be recovered from $(s_j)_{j \in A \ssminus \{ i \} }$
by $s_i = - \sum\limits_{j \in A \ssminus \{ i \} } s_j $.
So for all $d \in \Z_{\geq 0}$ we have 
$H^d(U;\calS_{A,k}) \cong 
 \bigoplus\limits_{j \in A \ssminus \{ i \} } H^d(U;\calO_\frakL)$.
So, if $H^{\geq 1}(U;\calO_\frakL)$ vanishes,
then $H^{\geq 1}(U,\calS_{A,k})$ vanishes too.
This proves the conclusion of Theorem~\ref{VAk} in this case.
\end{proof}

\begin{Lemma}[\textbf{Getting from Theorem \ref{HA} to Theorem~\ref{DAk}}]
\labell{DAk-induction}
Fix a subset $A \subseteq \{1,\ldots,r\}$
and an integer $k$ such that $1 \leq k < |A|$.
Assume that the conclusion of Theorem~\ref{HA} is true 
for all proper subsets $A' \subsetneq A$.
Then the conclusion of Theorem~\ref{DAk} is true 
for the set $A$ and the integer $k$.
\end{Lemma}

\begin{proof}
Fix an open subset $U \subseteq \frakX$ 
such that $H^{\geq 1}(U;\calO_\frakL) = \{ 0 \}$.
Fix a collection of holomorphic sections,
$ s_{j,B'} \in \Gamma_{\hol} (U,\frakL)$, 
indexed by $j \in A$ and $B' \in \binom{A \ssminus \{ j \} }{k-1}$,
such that $\sum\limits_{ B \in \binom{A}{k} }
 \prod\limits_{\ell \in B} t_\ell 
 \sum\limits_{j \in B} s_{ j , B \ssminus \{ j \} } = 0$.
This equality implies that for any $B \in \binom{A}{k}$
the sum $\sum\limits_{j \in B} s_{j, B \ssminus \{ j \} } $
vanishes on $U \cap \bigcap\limits_{ i \in A \ssminus B } \{ t_i=0 \}$.
By Theorem~\ref{HA} for $A' := A \ssminus B$, 
we get a collection $s_{i,B} \in \Gamma_{\hol} (U,\frakL)$,
indexed by $i \in A \ssminus B$, such that
\[
 \sum\limits_{j \in B} s_{j, B \ssminus \{ j \} }
 = \sum\limits_{ i \in A \ssminus B} t_i s_{i,B},
\]
as required.
This proves the conclusion of Theorem~\ref{DAk} in this case.
\end{proof}

\bigskip 

The following lemma proves the conclusion of Theorem~\ref{VAk} 
for $H^d$ and for $A$ and $k$,
assuming Theorem~\ref{DAk} for $A$ and $k$ and some easier cases
of Theorem~\ref{VAk}.

\begin{Lemma}[\textbf{Getting from Theorem~\ref{DAk}
to an inductive step for Theorem~\ref{VAk}}]
\labell{VAk-induction}
Let $d \in \N$.
Fix a subset $A \subseteq \{1,\ldots,r\}$
and an integer $k$ such that $1 \leq k < |A|$.
Suppose that the conclusion of Theorem~\ref{DAk} is true 
for the set $A$ and the integer $k$.
Fix an open subset $U \subseteq \frakX$.
Assume that $H^d(U;\calO_{\frakL}) = \{ 0 \}$,
that $H^d(U;\calS_{A,k+1}) = \{ 0 \}$,
and that $H^{d+1} (U;\calS_{A',1}) = \{ 0 \}$
for every proper subset $A' \subsetneq A$.
Then $H^d(U;\calS_{A,k}) = \{ 0 \}$.
\end{Lemma}

\begin{proof}
Take a \v{C}ech $d$-cocycle of $\calS_{A,k}$ over $U$.
It is given by an open covering $\{ U_\alpha \}$ of $U$,
and collections of sections
$$s_{i,B'}^{\alpha_0,\ldots,\alpha_d} 
 \in \Gamma_\hol( U_{\alpha_0} \cap \ldots \cap U_{\alpha_d} , \frakL),$$
indexed by $i \in A$ and $B' \in \binom{A \ssminus \{ i \} }{k-1}$,
such that, when we fix $\alpha_0,\ldots,\alpha_d$ and vary $i$ and $B'$,
\begin{equation} \labell{in SAk}
\left( s_{i,B'} ^{\alpha_0,\ldots,\alpha_d} \right) 
 \in \calS_{A,k} (U_{\alpha_0}\cap\ldots\cap U_{\alpha_d}),
\end{equation}
and when we fix $i$ and $B'$ and vary $\alpha_0,\ldots,\alpha_d$,
\begin{equation} \labell{cocycle}
\delta \left( s_{i,B'} ^{\alpha_0,\ldots,\alpha_d} \right) = 0 ,
\end{equation}
where $\delta$ is the \v{C}ech differential.

After passing to a refinement, we may assume that,
for every $\alpha_0,\ldots,\alpha_d$ in the indexing set
for the covering, 
$H^{\geq 1}(U_{\alpha_0} \cap \ldots \cap U_{\alpha_d}; \calO_\frakL) 
 = \{ 0 \}$.

\medskip

For a moment, fix $\alpha_0,\ldots,\alpha_d$.
By Theorem~\ref{DAk} for the set $A$ and the integer~$k$,
and by~\eqref{in SAk},
there exists a collection of sections
$$s_{i,B}^{\alpha_0,\ldots,\alpha_d} 
 \in \Gamma_\hol( U_{\alpha_0} \cap \ldots \cap U_{\alpha_d} , \frakL),$$
indexed by $i \in A$ and $B \in \binom{A \ssminus \{ i \}}{k}$,
such that, for every $B \in \binom{A}{k}$,
\begin{equation} \labell{choice}
 \sum_{j \in B} s_{ j , B \ssminus \{ j \} }^{\alpha_0,\ldots,\alpha_d}
 = \sum_{i \in A \ssminus B} t_i s_{i,B}^{ \alpha_0,\ldots,\alpha_d } .
\end{equation}
By Remark~\ref{DAk-remark}, 
\[
\left( s_{i,B} ^{\alpha_0,\ldots,\alpha_d} \right)
 \in \calS_{A,k+1}(U_{\alpha_0} \cap\ldots\cap U_{\alpha_d}).
\]

\medskip

Now, we let $\alpha_0,\ldots,\alpha_d$ vary. 
The collection $\left( s_{i,B} ^{\alpha_0,\ldots,\alpha_d} \right)$
gives a $d$-cochain in $\calS_{A,k+1}$, 
which might not be a cocycle.  Let
\begin{equation} \labell{delta}
 (\sigma_{i,B}^{\alpha_0,\ldots,\alpha_{d+1}})
:= \delta \left( ( s_{i,B}^{\alpha_0,\ldots,\alpha_d} ) \right) .
\end{equation}

\medskip

For a moment, fix $B \in \binom{A}{k}$. We have 
\begin{equation} \labell{sum0}
\left(
 \sum_{i \in A \ssminus B} t_i \sigma_{i,B}^{\alpha_0,\ldots,\alpha_{d+1}}
\right)
 = \delta \left( 
   \sum_{i \in A \ssminus B} t_i s_{i,B}^{\alpha_0,\ldots\alpha_d} \right)
 = \sum_{j \in B} \delta \left( 
           s_{j, B \ssminus \{ j \} }^{\alpha_0,\ldots,\alpha_d} \right)
 = 0,
\end{equation}
where the first equality is by~\eqref{delta},
the second equality is by~\eqref{choice},
and the vanishing at the end is by~\eqref{cocycle}.
By~\eqref{sum0}, for each $\alpha_0,\ldots,\alpha_{d+1}$,
the collection of sections $(\sigma_{i,B} ^{\alpha_0,\ldots,\alpha_{d+1}})$,
indexed by $i \in A \ssminus B$, gives an element of 
$\calS_{A \ssminus B, 1}(U_{\alpha_0}\cap\ldots\cap U_{d+1})$.
As $\alpha_0,\ldots,\alpha_{d+1}$ vary,
we get a $(d+1)$-cochain in $\calS_{A \ssminus B , 1}$.
As a consequence of \eqref{delta}, this cochain is a cocycle.
Because $H^{d+1}(U;\calS_{A \ssminus B , 1}) = \{ 0 \}$,
after possibly passing to a refinement of the covering,
there exists a $d$-cochain of $\calS_{A \ssminus B, 1}$, given by sections 
$\sigma_{i,B} ^{\alpha_0,\ldots,\alpha_d} 
 \in \Gamma_\hol(U_{\alpha_0}\cap\ldots\cap U_{\alpha_d} , \frakL)$
indexed by $i \in A \ssminus B$, such that,
when we fix $\alpha_0,\ldots,\alpha_d$ and $B$ and vary $i$,
\begin{equation} \labell{in S A-B 1}
\left( \sigma_{i,B}^{\alpha_0,\ldots,\alpha_d} \right)
 \in \calS_{A \ssminus B, 1} (U_{\alpha_0} \cap \ldots \cap U_{\alpha_d}),
\end{equation}
and when we fix $i$ and $B$ and vary $\alpha_0,\ldots,\alpha_d$,
\begin{equation}\labell{primitive}
\delta \left( \sigma_{i,B} ^{\alpha_0,\ldots,\alpha_d} \right)
 = \left( \sigma_{i,B} ^{\alpha_0,\ldots,\alpha_{d+1}} \right).
\end{equation}
In \eqref{in S A-B 1},
being in $\calS_{A \ssminus B, 1}$ means that

\begin{equation} \labell{vanish}
\sum\limits_{i \in A \ssminus B} t_i \sigma_{i,B} ^{\alpha_0,\ldots,\alpha_d} 
 = 0.
\end{equation}

\medskip

We now let $B \in \binom{A}{k}$ vary.  We claim that the collection
of sections $\sigma_{i,B} ^{\alpha_0,\ldots,\alpha_d}$,
indexed by $i \in A$ and $B \in \binom{A \ssminus \{ i \}}{k}$,
is in $\calS_{A,k+1} (U_0 \cap \ldots \cap U_{\alpha_d})$.
Indeed,
\[
 \sum\limits_{\wt{B} \in \binom{A}{k+1} }
 \prod_{\ell \in \wt{B}} t_\ell 
 \sum\limits_{j \in \tilde{B}} \sigma_{j , \wt{B} \ssminus \{ j \} } 
 ^{\alpha_0,\ldots,\alpha_d}
 = \sum\limits_{B \in \binom{A}{k}} \prod\limits_{\ell \in B} t_\ell
 \sum\limits_{j \in A \ssminus B} t_j \sigma_{j,B} ^{\alpha_0,\ldots,\alpha_d}
 = 0,
\]
where the first equality is obtained by taking $B=\wt{B}\ssminus \{ j \}$,
and the vanishing at the end is by~\eqref{vanish}.

Thus, $\left( s_{i,B} ^{\alpha_0,\ldots,\alpha_d} \right)$
and $\left( \sigma_{i,B} ^{\alpha_0,\ldots,\alpha_d} \right)$
are both $d$-cochains of $\calS_{A,k+1}$.
By \eqref{delta} and \eqref{primitive}, they have the same coboundary.
So
$ \left( s_{i,B} ^{\alpha_0,\ldots,\alpha_d}  
 - \sigma_{i,B} ^{\alpha_0,\ldots,\alpha_d} \right) $
is a $d$-cocycle of $\calS_{A,k+1}$.  
Because $H^d(U;\calS_{A,k+1}) = \{0\}$, there exists
a $(d-1)$-cochain
$ \left( {\wh{s}}_{i,B} ^{\alpha_0,\ldots,\alpha_{d-1}} \right)$ 
of $\calS_{A,k+1}$ such that 
\begin{equation} \labell{d-1}
 \delta \left( {\wh{s}}_{i,B} ^{\alpha_0,\ldots,\alpha_{d-1}} \right)
 = 
   \left( s_{i,B} ^{\alpha_0,\ldots,\alpha_d}  
 - \sigma_{i,B} ^{\alpha_0,\ldots,\alpha_d} \right). 
\end{equation}
For later reference we recall that being in $\calS_{A,k+1}$ means that, 
for every $\alpha_0,\ldots,\alpha_{d-1}$,
\begin{equation} \labell{SA kplus1}
\sum\limits_{\wt{B} \in \binom{A}{k+1}} \prod\limits_{\ell \in \wt{B}} t_\ell
 \sum\limits_{i \in \wt{B}} 
 \wh{s}_{ i , \wt{B} \ssminus \{ i \} } ^{\alpha_0,\ldots,\alpha_{d-1}} = 0.
\end{equation}

For every $B \in \binom{A}{k}$,
\begin{equation} \labell{delta'}
\delta \left( 
 \sum_{i \in A \ssminus B} t_i {\wh{s}}_{i,B} ^{\alpha_0,\ldots,\alpha_{d-1}}
 \right)
 = \left( 
 \sum_{i \in A \ssminus B} t_i s_{i,B} ^{\alpha_0,\ldots,\alpha_{d}}
 -
 \sum_{i \in A \ssminus B} t_i \sigma_{i,B} ^{\alpha_0,\ldots,\alpha_{d}}
   \right)  
 = \sum_{j \in B} s_{ j , B \ssminus \{ j \} } ^{\alpha_0\ldots,\alpha_d}, 
\end{equation}
where the first equality is by~\eqref{d-1}
and the second equality is by~\eqref{choice} and~\eqref{vanish}.

\medskip

By~\eqref{cocycle}, 
for every $B \in \binom{A}{k}$ and $j \in B$, 
the collection 
$\left( s_{j, B \ssminus \{ j \} } ^{\alpha_0,\ldots,\alpha_d} \right)$
is a $d$-cocycle of $\calO_\frakL$.  
Because $H^d( U; \calO_\frakL ) = \{ 0 \}$, there exist
$s_{ j , B \ssminus \{ j \} } ^{\alpha_0,\ldots,\alpha_{d-1}}
 \in \Gamma_\hol (U_{\alpha_0}\cap\ldots\cap U_{\alpha_{d-1}},\frakL)$ 
such that 
\begin{equation} \labell{primitive'}
\delta \left( 
   s_{ j , B \ssminus \{ j \} } ^{\alpha_0,\ldots,\alpha_{d-1}} \right)
 = \left( 
   s_{ j , B \ssminus \{ j \} } ^{\alpha_0,\ldots,\alpha_d} \right) .
\end{equation}

For any $B \in \binom{A}{k}$ and $j \in B$,
if $j \neq \min B$ then we fix an arbitrary $(d-1)$-cochain 
$\left( s_{ j , B \ssminus \{ j \} } ^{\alpha_0,\ldots,\alpha_{d-1}} \right)$
of $\calO_\frakL$
that satisfies~\eqref{primitive'},
and if $j = \min B$ then we take
\begin{equation} \labell{define}
 s_{ j , B \ssminus \{ j \} } ^{\alpha_0,\ldots,\alpha_{d-1}}
 := \sum_{i \in A \ssminus B} 
    t_i {\wh{s}} _{i,B} ^{\alpha_0,\ldots,\alpha_{d-1}}
     - \sum_{\substack{ j' \in B \\ j' \neq \min B}} 
       s_{ j' , B \ssminus \{ j' \}} ^{\alpha_0,\ldots,\alpha_{d-1}} 
\quad \text{ when $j = \min B$}.
\end{equation}
The $(d-1)$-cochain of $\calO_\frakL$ given by~\eqref{define} 
also satisfies~\eqref{primitive'};
this follows from the equation~\eqref{delta'}
and from the equations~\eqref{primitive'} with $j' \neq \min B$.

To complete the proof, we show that the sections 
$\left( s_{j, B \ssminus \{ j \}} ^{\alpha_0,\ldots,\alpha_{d-1}} \right)$
give a $(d-1)$-cochain of $\calS_{A,k}$
whose boundary is the given $d$-cocycle of $\calS_{A,k}$.
Indeed, these sections give a cochain of $\calS_{A,k}$, because
$$ \sum_{ B \in \binom{A}{k} } \prod_{\ell \in B} t_\ell
   \sum_{ j \in B } s_{ j , B \ssminus \{ j \} } ^{\alpha_0,\ldots,\alpha_{d-1}}
 = \sum_{ B \in \binom{A}{k} } \prod_{\ell \in B} t_\ell
   \sum_{i \in A \ssminus B} 
               t_i {\wh{s}}_{i,B} ^{\alpha_0,\ldots,\alpha_{d-1}} $$
$$ = \sum_{ \wt{B} \in \binom{A}{k+1} } \prod_{\ell \in \wt{B}} t_\ell
 \sum_{i \in \wt{B}} 
   {\wh{s}} _{i , \wt{B} \ssminus \{ i \} } ^{\alpha_0,\ldots,\alpha_{d-1}}
 = 0, $$
where the first equality is by~\eqref{define},
the second equality is obtained by setting $\wt{B} = B \sqcup \{ i \}$,
and the last vanishing is by~\eqref{SA kplus1}.
And the coboundary of this $(d-1)$-cochain is the given $d$-cocycle,
by~\eqref{primitive'}.
\end{proof}

\begin{Corollary}[\textbf{Theorem~\ref{DAk} implies Theorem~\ref{VAk}}]
\labell{DAk-implies-VAk}
Fix a subset $A \subseteq \{1, \ldots, r\}$.
Assume that the conclusion of Theorem~\ref{DAk} is true for the set $A$
and for every integer $k$ such that $1 \leq k \leq |A|$,
and assume that the conclusion of Theorem~\ref{VAk} is true 
for every proper subset $A' \subsetneq A$
and every integer $k'$ such that $1 \leq k' \leq |A'|$.
Then the conclusion of Theorem~\ref{VAk} is true 
for the set $A$ and for every integer $k$ such that $1 \leq k \leq |A|$.
\end{Corollary}

\begin{proof}
Fix an open subset $U \subseteq \frakX$
such that $H^{\geq 1}(U; \calO_\frakL) = \{ 0 \}$.
We would like to show that 
$H^{\geq 1}(U; \calS_{A,k}) = \{ 0 \}$ for all $1 \leq k \leq |A|$.

We argue by decreasing induction on $k$. 
When $k = |A|$, by Lemma~\ref{DAk-VAk-base},
$H^{\geq 1}(U; \calS_{A,k}) {= \{ 0 \}}$.
Now assume that $1 \leq k<|A|$ and that 
$H^{\geq 1}(U; \calS_{A,k+1}) {= \{ 0 \}}$.
We would like to show that $H^{\geq 1}(U; \calS_{A,k}) {= \{ 0 \}}$.

Fix any $d \in \N$.
By assumption, the conclusion of Theorem~\ref{DAk}
is true for the set $A$ and for the integer $k$.
By assumption, $H^{\geq 1}(U;\calO_\frakL) = \{ 0 \}$;
in particular, $H^{d}(U;\calO_\frakL) = \{ 0 \}$.
By the induction hypotheses for $k$, 
we have $H^{\geq 1}(U; \calS_{A,k+1}) {= \{ 0 \}}$;
in particular, $H^{d}(U; \calS_{A,k+1}) {= \{ 0 \}}$.
By assumption, the conclusion of Theorem~\ref{VAk} is true 
for every proper subset $A' \subsetneq A$
and every integer $k'$ such that $1 \leq k' \leq |A'|$;
in particular, $H^{d+1}(U; \calS_{A',1}) = \{ 0 \}$
for every proper subset $A' \subsetneq A$ and $k'=1$.
By Lemma~\ref{VAk-induction}, we conclude 
that $H^{d}(U; \calS_{A,k}) = \{ 0 \}$, as required.
\end{proof}

\begin{Lemma}[\textbf{Base case of Theorem~\ref{HA}}]
\labell{HA-base}
The conclusion of Theorem~\ref{HA} is true when $|A|=1$. Namely: 

Fix $j \in \{1, \ldots, r\}$.
For any open subset $U \subseteq \frakX$
such that $H^{\geq 1} (U,\calO_\frakL) = \{ 0 \}$
and any section $s \in \Gamma_\hol(U,\frakL)$,
if $s$ vanishes on $U \cap \{ t_j=0 \}$,
then there exists a section $s_j \in \Gamma_{\hol}(U,\frakL)$
such that $s=t_js_j$.
\end{Lemma}

\begin{proof}
By taking local Taylor expansions, 
there exists an open covering $\{ U_\alpha \}$ of $U$
and, for each $\alpha$, 
a section $s_j^{\alpha} \in \Gamma_\hol(U_\alpha,\frakL)$,
such that ${s|}_{U_\alpha} = t_j s_j^{\alpha}$.
Every two sections $s_j^\alpha$ and $s_j^{\alpha'}$
coincide on $U_\alpha \cap U_{\alpha'} \cap \{ t_j \neq 0\}$
(because they are both equal there to $\frac{1}{t_j} s$);
because $\{ t_j \neq 0 \}$ is open and dense,
they coincide on the entire overlap $U_\alpha \cap U_{\alpha'}$.
So the sections $s_j^\alpha$
fit together into a section $s_j \in \Gamma_\hol(U,\frakL)$,
which has the required property.
\end{proof}

\begin{Lemma}[\textbf{Inductive step for Theorem~\ref{HA}}]
\labell{HA-induction}
Fix a subset $A \subseteq \{1,\ldots,r\}$.
Assume that  the conclusion of Theorem~\ref{VAk} for the first cohomology $H^1$ is true for    the set $A$ and the integer $k=1$, namely,  for any open subset $U\subseteq \frakX$, if $H^{\ge 1}(U; \calO_{\frakL})=\{0\}$, then  $H^1(U; \calS_{A, 1})=\{0\}$.
Then the conclusion of Theorem~\ref{HA} is true for the set~$A$.
\end{Lemma}

\begin{proof}
Fix an open subset $U$ of $\frakX$
such that $H^{\geq 1}(U;\calO_{\frakL}) = \{ 0 \} $.
Fix a holomorphic section $s \in \Gamma_\hol(U,\frakL)$
that vanishes on $U \cap \bigcap\limits_{i \in A} \{ t_i = 0\}$.
By Taylor expansions, there exists an open covering $\{ U_\alpha \}$ of $U$
and, for each $\alpha$,
sections $s^{\alpha}_i \in \Gamma_\hol(U_\alpha)$,
indexed by $i \in A$,
such that 
$$ {s|}_{U_\alpha} = \sum\limits_{i \in A} t_i s^{\alpha}_i.$$ 
For every $\alpha$ and $\alpha'$, consider
\begin{equation} \labell{sigma alpha alpha'}
 \sigma^{\alpha,\alpha'} _i :=
   s^{\alpha}_i - s^{\alpha'}_i
 \in \Gamma_\hol \left( U_\alpha \cap U_{\alpha'} , \frakL \right).
\end{equation}
Then
\[
    \sum\limits_{i \in A} t_i \sigma^{\alpha,\alpha'}_i = 0.
\]
This condition exactly means that the collection of sections
$\left( \sigma^{\alpha,\alpha'}_i \right)$,
indexed by $i \in A$,
is in $\calS_{A,1}(U_\alpha \cap U_{\alpha'})$.
Moreover, as a consequence of~\eqref{sigma alpha alpha'},
$\left( \sigma^{\alpha,\alpha'}_{i} \right)$
is a 1-cocycle for the sheaf $\calS_{A,1}$.
By the conclusion of  Theorem~\ref{VAk} for  the first cohomology~$H^1$ with the set $A$ and the integer $k=1$, 
after possibly passing to a refinement of our covering,
we get $\sigma^{\alpha}_i \in \Gamma_\hol(U_\alpha)$,
indexed by $i \in A$,
such that $\sum\limits_{i \in A} t_i \sigma^\alpha_i = 0$,
and such that 
$\sigma^{\alpha,\alpha'}_i = \sigma^\alpha_i - \sigma^{\alpha'}_i$.
Then $( s^{\alpha}_i - \sigma^{\alpha}_i )$
agree on overlaps, and 
\[
    {s|}_{U_\alpha} = \sum\limits_{i \in A} 
    t_i \left( s^{\alpha}_i - \sigma^\alpha_i  \right) .
\]
\end{proof}

\bigskip

\begin{proof}[Proof of Theorems~\ref{HA}, \ref{DAk}, and \ref{VAk}.]

Fix a subset $A \subseteq \{1,\ldots,r\}$.   We argue by induction on $|A|$. 

The base case is $k=|A|=1$.
In this case, the conclusion of Theorem~\ref{HA} is true 
by Lemma~\ref{HA-base},
and the conclusions of Theorems~\ref{DAk} and~\ref{VAk} 
are true by Lemma~\ref{DAk-VAk-base}.

Now suppose that the conclusions of Theorems~\ref{HA} and~\ref{VAk} are true 
for every proper subset $A' \subsetneq A$ 
and for every integer $k'$ such that $1 \leq k' \leq |A'|$.

When $k=|A|$, the conclusion of Theorem~\ref{DAk} is true for $A$ and $k$
by Lemma~\ref{DAk-VAk-base}.
When $1 \leq k < |A|$, the conclusion of Theorem~\ref{DAk} 
is true for $A$ and $k$ by Lemma~\ref{DAk-induction}.
Thus, the conclusion of Theorem~\ref{DAk} is true for the set $A$ 
and for every integer $k$ such that $1 \leq k \leq |A|$.

By Corollary~\ref{DAk-implies-VAk}, we conclude that
the conclusion of Theorem~\ref{VAk} for the first cohomology $H^1$, 
and for every integer $k$ such that $1 \leq k \leq |A|$.

In particular, we have shown that the conclusion of Theorem~\ref{VAk} 
for the first cohomology $H^1$
is true for the set $A$ and the integer $k=1$.
By Lemma~\ref{HA-induction}, we obtain 
that the conclusion of Theorem~\ref{HA} is true for the set $A$.
\end{proof}

\section{A ``linear independence lemma''}
\labell{app:indep}

The purpose of this appendix is to prove Lemma~\ref{one to one},
which we use in the proof of Proposition~\ref{p:one-to-one}.

Fix a natural number $r$,
a complex manifold $\frakX$, a submersion 
$$\pi = (t_1,\ldots,t_r) \colon \frakX \to \C^r,$$
and a holomorphic line bundle 
$$\frakL \to \frakX.$$
Denote $\frakX_\reg = \pi^{-1}( (\Cx)^r )$
and $X_\bfzero = \pi^{-1}(\{ 0 \})$.

\begin{Lemma}\labell{one to one}
Let $\vell^{(1)},\ldots,\vell^{(m)}$ be distinct elements of $\Z^r$.
Let $\sigma_1,\ldots,\sigma_m \in \Gamma_{\hol}(\frakX,\frakL)$.
Suppose that
$$ \bft^{\vell^{(1)}} \sigma_1 + \ldots + \bft^{\vell^{(m)}} \sigma_m = 0
\ \text{ on } \ \frakX_\reg. $$
Then there exists $j \in \{ 1,\ldots,m\}$ 
such that $\sigma_j$ vanishes on $X_\bfzero$.
\end{Lemma}

Lemma~\ref{one to one} is the special case
of the following more general Lemma
when the set $A$ is all of $\{ 1, \ldots, r \}$:

\begin{Lemma}\labell{one to one A}
Let $A \subseteq \{1 , \ldots, r \}$ and let $m \in \N$.
Let $\vell^{(1)},\ldots,\vell^{(m)}$ be distinct elements of~$\Z^A$.
Let $\sigma_1,\ldots,\sigma_m \in \Gamma_{\hol}(\frakX,\frakL)$.
Suppose that 
$$ \bft^{\vell^{(1)}} \sigma_1 + \ldots + \bft^{\vell^{(m)}} \sigma_m 
\ \text{ vanishes on } \ 
  \left( \cap_{i \in A} \{ t_i \neq 0 \} \right) \cap 
  \left( \cap_{i \not\in A} \{ t_i = 0\} \right). $$
Then there exists $j \in \{ 1,\ldots,m \}$ 
such that $\sigma_j$ vanishes on $X_\bfzero$.
\end{Lemma}

\begin{Remark}
In the setup of Lemma~\ref{one to one}, the section
$ \bft^{\vell^{(1)}} \sigma_1 + \ldots + \bft^{\vell^{(m)}} \sigma_m $
is defined on the set $\frakX_\reg$
and might not be defined outside it.
In the setup of Lemma~\ref{one to one A}, the section
$ \bft^{\vell^{(1)}} \sigma_1 + \ldots + \bft^{\vell^{(m)}} \sigma_m $
is defined on the set $ \cap_{i \in A} \{ t_i \neq 0 \} $
and might not be defined outside it.
\end{Remark}

\begin{proof}[Proof of Lemma~\ref{one to one A} when $A$ is empty
or $m=1$.] 
When $A$ is empty, $\Z^A = \{0\}$, 
so the vectors $\vell^{(j)}$ are all zero,
so $m=1$ (otherwise the vectors $\vell^{(j)}$ would not be distinct).
The assumption is that $\sigma_1$ vanishes on $X_\bfzero$.
But this is the same as what we need to prove.

When $m=1$, we are assuming that $\bft^{\vell^{(1)}} \sigma_1$
vanishes on the intersection 
$ \left( \cap_{i \in A} \{ t_i \neq 0 \} \right) \cap 
  \left( \cap_{i \not\in A} \{ t_i = 0\} \right)$.
Multiplying by $\bft^{-\vell^{(1)}}$,
we obtain that $\sigma_1$ also vanishes on this intersection.
Because $X_\bfzero$ is in the closure of this intersection,
by continuity, $\sigma_1$ vanishes on~$X_\bfzero$.
\end{proof}

\begin{proof}[Proof of Lemma~\ref{one to one A}]
We proceed by induction.
Assume that $A$ is non-empty and that $m$ is greater than $1$,
and assume that Lemma~\ref{one to one A} is true
for all pairs $(A',m')$ 
of a proper subset $A' \subsetneq A$ and a smaller natural number $m' < m$.

After possibly permuting the coordinates, assume that $1 \in A$
and that the first coordinates 
of $\vell^{(1)}, \ldots, \vell^{(m)}$ are not all equal.  
Let $\ell_1$ be the minimum of the values of these first coordinates.  
For each $j$, let $\vell^{(j)}{}'$ be obtained from $\vell^{(j)}$ 
by subtracting $\ell_1$ from the first coordinate.
After possibly reordering the vectors, assume that 
the first coordinate of $\ell^{(j)}{}'$ is zero for $j = 1, \ldots, m'$
and is positive for $j = m'+1, \ldots, m$.
Let 
$$ A' := A \ssminus \{ 1 \}.$$
View $\Z^{A'}$ as the subset of $\Z^A$ where the first coordinate is zero.

The sections $t^{\vell^{(j)}{}'} \sigma_j$
are defined (at least) on $ \cap_{i \in A'} \{ t_i \neq 0 \} $,
because the first coordinates of 
$\vell^{(1)}{}', \ldots, \vell^{(m)}{}'$ are non-negative.
By assumption, 
$t_1^{\ell_1} \left(
 \bft^{\vell^{(1)}{}'} \sigma_1 + \ldots + \bft^{\vell^{(m)}{}'} \sigma_m 
 \right)$
vanishes on the set
$ \left( \cap_{i \in A} \{ t_i \neq 0 \} \right) \cap 
  \left( \cap_{i \not\in A} \{ t_i = 0\} \right)$.
Because $t_1^{\ell_1}$ is nonvanishing on this set, the sum 
$\bft^{\vell^{(1)}{}'} \sigma_1 + \ldots + \bft^{\vell^{(m)}{}'} \sigma_m$
also vanishes on this set.
By continuity, this sum vanishes on 
$ \left( \cap_{i \in A'} \{ t_i \neq 0 \} \right) \cap 
  \left( \cap_{i \not\in A} \{ t_i = 0\} \right)$.

Because the first coordinate of $\vell^{(j)}$ is strictly positive
for $j = m'+1,\ldots,m$,
the corresponding summands $t^{\vell^{(j)}{}'} \sigma_j$ 
vanish whenever $t_1=0$, so they vanishes on 
$ \left( \cap_{i \in A'} \{ t_i \neq 0 \} \right) \cap \{ t_1 = 0\} $.

Because the sum
$ t^{\vell^{(1)}{}'} \sigma_1 + \ldots + t^{\vell^{(m)}{}'} \sigma_m $
vanishes on the set
$ \left( \cap_{i \in A'} \{ t_i \neq 0 \} \right)   \cap
  \left( \cap_{i \not\in A} \{ t_i = 0\} \right)$,
and the partial sum 
$\sum_{j=m'+1}^m t^{\vell^{(j)}{}'} \sigma_j$
vanishes on the set
$ \left( \cap_{i \in A'} \{ t_i \neq 0 \} \right) \cap \{ t_1 = 0\} $,
the remaining partial sum
$\bft^{\vell^{(1)}{}'} \sigma_1 + \ldots + \bft^{\vell^{(m')}{}'} \sigma_{m'}$
vanishes on the intersection of these two sets, which is 
$ \left( \cap_{i \in A'} \{ t_i \neq 0 \} \right)  \cap
  \left( \cap_{i \not\in A'} \{ t_i = 0\} \right)$.

By Lemma~\ref{one to one A} for the proper subset $A' \subsetneq A$
and the smaller integer $m' < m$,
we conclude that at least one of $\sigma_1,\ldots,\sigma_{m'}$
vanishes on $X_\bfzero$.

This completes the proof of Theorem~\ref{one to one A} for the pair $(A,m)$.
\end{proof}


\smallskip\noindent\textbf{Acknowledgements.}
Yael Karshon is deeply grateful to Michael Grossberg
for introducing her to his work with Raoul Bott,
and to Joseph Bernstein for an inspiring conversation,
back in the early/mid 1990s.
More recently, Jihyeon Jessie Yang is grateful to Dave Anderson 
for introducing her to test configurations.
We are grateful for helpful discussions 
with Dave Anderson, Michel Brion, Megumi Harada, Johannes Hofscheier, 
David Kazhdan, Kiumars Kaveh, Allen Knutson, Jiang-Hua Lu.
Jihyeon Jessie Yang is grateful for the hospitality 
and funding by Professor Jiang-Hua Lu at the University of Hong Kong.
This work was partially supported by the
Natural Sciences and Engineering Research Council of Canada
and by the Fields Institute for Research in Mathematical Sciences.



\end{document}